\documentclass[11pt]{amsart}

\usepackage{graphicx,amssymb}
\usepackage[all]{xy}
\SelectTips{cm}{}
\usepackage[dvipsnames]{xcolor}

\usepackage{hyperref}
\hypersetup{colorlinks=true,citecolor=brown,linkcolor=blue,urlcolor=black,filecolor=black}

\allowdisplaybreaks
\numberwithin{equation}{section}

\theoremstyle{plain}
\newtheorem{theorem}{Theorem}[section]
\newtheorem{lemma}[theorem]{Lemma}
\newtheorem{proposition}[theorem]{Proposition}

\theoremstyle{definition}

\newtheorem{example}[theorem]{Example}

\newtheorem{remark}[theorem]{Remark}

\newtheorem{definition}[theorem]{Definition}

\newcommand{\up}{\vspace{-0.5cm}}

\newcommand{\Alex}{\mathsf{A}}
\newcommand{\Reid}{\mathsf{R}} 

\newcommand{\C}{\mathcal{C}}
\newcommand{\Cob}{\mathsf{Cob}}
\newcommand{\mcg}{\mathcal{M}}
\newcommand{\Tang}{\mathsf{Tang}}
\newcommand{\TCob}{\mathsf{TangCob}}
\newcommand{\PMod}{\mathsf{grMod}}
\newcommand{\PVect}{\mathsf{grVect}}

\newcommand{\Aut}{\operatorname{Aut}}
\newcommand{\Hom}{\operatorname{Hom}}
\newcommand{\Id}{\operatorname{Id}}
\newcommand{\vol}{\omega}
\newcommand{\Ker}{\operatorname{Ker}}

\newcommand{\Tors}{\operatorname{Tors}}
\newcommand{\rank}{\operatorname{rank}}

\newcommand{\mcyl}{\mathbf{c}}

\newcommand{\Q}{\mathbb{Q}}
\newcommand{\Z}{\mathbb{Z}}
\newcommand{\F}{\mathbb{F}}
\newcommand{\N}{\mathbb{N}}

\title[]{A functorial extension of the abelian\\ Reidemeister torsions of three-manifolds}

\date{October 23, 2014}
 
\author[]{Vincent Florens} 
\address{Vincent Florens \newline
\indent  LMA, Universit\'e de Pau \& CNRS \newline
\indent Avenue de l'Universit\'e \newline 
\indent  64000 Pau,  France \newline 
\indent  $\mathtt{vincent.florens@univ\hbox{-}pau.fr}$
} 

\author[]{Gw\'ena\"el Massuyeau}
\address{Gw\'ena\"el Massuyeau \newline
\indent IRMA,    Universit\'e de Strasbourg \& CNRS \newline
\indent 7 rue Ren\'e Descartes \newline
\indent 67084 Strasbourg, France \newline
\indent $\mathtt{massuyeau@math.unistra.fr}$
}

\begin{document}

\begin{abstract}
  Let $\F$ be a field and let $G\subset \F\setminus \{0\}$ be a multiplicative subgroup.
We consider the category $\Cob_G$ of $3$-dimensional cobordisms  equipped with a representation of their fundamental group in  $G$, 
and the category $\mathsf{Vect}_{\F,\pm G}$ of  $\F$-linear maps  defined up to multiplication by an element of $\pm G$.
Using the elementary theory of  Reidemeister torsions, we construct  a ``Reidemeister functor'' from  $\Cob_G$  to $\mathsf{Vect}_{\F,\pm G}$.
In particular, when the group $G$ is  free abelian  and $\F$ is the field of fractions of  the group ring $\Z[G]$,
we obtain a functorial formulation of an Alexander-type invariant introduced by Lescop for $3$-manifolds with boundary;
when $G$ is trivial, the Reidemeister functor specializes to the TQFT developed by Frohman and Nicas to enclose the Alexander polynomial of knots.
The study of the Reidemeister functor is carried out for any multiplicative subgroup  $G\subset \F\setminus \{0\}$.
We obtain a duality result and we show that the resulting projective representation of the monoid of homology cobordisms 
is equivalent to the Magnus representation combined with the relative Reidemeister torsion.
\end{abstract}

\maketitle

\setcounter{tocdepth}{1}
{\footnotesize \tableofcontents}

\section{Introduction}

Let $\Cob$  be the category of $3$-dimensional cobordisms introduced by Crane and Yetter \cite{CY}, and whose definition we briefly recall.
The objects of $\Cob$  are integers $g\geq 0$, and correspond to compact connected oriented surfaces $F_g$ of genus $g$ with one boundary component. 
Indeed,  we fix for every $g\geq 0$ a model surface $F_g$ whose boundary is identified with $S^1$,
and we also fix a base point $\star$ on $\partial F_g=S^1$.
The morphisms $g_- \to g_+$ in the category $\Cob$ are the equivalence classes of cobordisms between the surfaces $F_{g_-}$ and $F_{g_+}$.
To be more specific, a \emph{cobordism} from $F_{g_-}$ to $F_{g_+}$ is a pair $(M,m)$ consisting of  a compact connected oriented $3$-manifold $M$
and  an orientation-preserving homeomorphism $m: F({g_-},{g_+}) \to \partial M$ where 
$$
F({g_-},{g_+}) := {-F_{g_-} \cup_{S^1 \times \{-1\}} \big( S^1 \times [-1,1] \big) \cup_{S^1 \times \{1\}} F_{g_+} };
$$
two such pairs $(M, m)$ and $(M^\prime, m^\prime)$ are \emph{equivalent} if there exists a homeomorphism 
$f: M  \to M^\prime$ such that $m^\prime  = f|_{\partial M} \circ m$.  
We shall  denote a pair $(M,m)$ simply by the upper-case letter $M$,  with the convention that the boundary-parametrization is always denoted by the lower-case letter $m$; 
besides, we denote by  $m_\pm: F_{g_\pm} \to M$ the restriction of $m$ composed with the inclusion of $\partial M$ into $M$. 
Thus   the cobordism $M$ ``runs'' from the \emph{bottom surface} $\partial_- M:= m_-(F_{g_-})$ to the  \emph{top surface} $\partial_+ M:= m_+(F_{g_+})$. 
The \emph{degree} of  the cobordism $M$ is the integer $g_+-g_-$.  

The composition $N\circ M$ of two cobordisms $M,N$  in $\Cob$ is defined by identifying $\partial_+ M$ to $\partial_- N$ and,
for any integer $g\geq 0$, the identity of the object $g$  is the cylinder $F_g \times [-1,1]$ with the obvious boundary-parametrization.
Our model surfaces $F_0,F_1,F_2,\dots$ also come with an identification of the boundary-connected sum 
$F_g\sharp_\partial F_h$ with the surface $F_{g+h}$ for any  $g,h\geq 0$. 
Thus the category $\Cob$ is enriched with a monoidal structure~$\otimes$:
the tensor product $g\otimes h$ of two integers $g,h$ is the sum $g+h$,
and the tensor product $M \otimes N$ of two cobordisms $M,N$ is their boundary-connected sum $M \sharp_\partial  N$.

Let now $G$ be an abelian group.
The category $\Cob$ can be refined to the category $\Cob_G$ of cobordisms equipped with a representation of the first integral homology group in $G$. 
To be more specific, an object of $\Cob_G$  is a pair $(g,\varphi)$ consisting of an integer $g\geq 0$ and a group homomorphism  $\varphi : H_1(F_g;\Z) \to G$. 
A morphism $(g_-,\varphi_-) \to (g_+,\varphi_+)$ in the category $\Cob_G$ is a pair $(M,\varphi)$ where $M\in \Cob(g_-,g_+)$
and $\varphi : H_1(M;\Z) \to G$ is a group homomorphism such that $\varphi\circ m_{\pm,*}=\varphi_{\pm}$. 
The composition of two morphisms $(M,\varphi)\in \Cob_G((g_-,\varphi_-),(g_+,\varphi_+))$ 
and $(N,\psi)\in \Cob_G((h_-,\psi_-),(h_+,\psi_+))$, such that $(g_+,\varphi_+)= (h_-,\psi_-)$, is defined by
$$
(N,\psi) \circ (M,\varphi)   := (N \circ M,  \psi+\varphi)
$$
where $N\circ M$ is the composition in $\Cob$ and $\psi + \varphi: H_1(N\circ M;\Z) \to G$ is defined from $\varphi$ and $\psi$ by using the Mayer--Vietoris theorem. 
The monoidal structure of $\Cob$ also extends to the category $\Cob_G$: the tensor product of objects is 
$$
(g,\varphi) \otimes (h,\psi) := (g+h,\varphi \oplus \psi)
$$
where $H_1(F_{g+h};\Z)= H_1(F_g \sharp_\partial F_h;\Z)$ is identified with $H_1(F_g;\Z)\oplus H_1(F_h;\Z)$, 
and the tensor product of morphisms is 
$$
(M,\varphi) \otimes (N,\psi) :=  (M \sharp_\partial N, \varphi \oplus \psi)
$$
where $H_1(M \sharp_\partial N;\Z)$ is identified with $H_1(M;\Z)\oplus H_1(N;\Z)$.
 
Consider now  a commutative ring  $R$ and fix a subgroup $G \subset R^\times$  of its group of units.
Let $\PMod_{R,\pm G}$ be the category whose objects are $\Z$-graded $R$-modules 
and whose morphisms are  graded $R$-linear maps of arbitrary degree, up to multiplication by an element of $\pm G$. 
The usual tensor product  of graded $R$-modules defines  a monoidal structure on the category $\PMod_{R,\pm G}$:
here the tensor product $a \otimes b$ of two graded $R$-linear maps $a:U \to U'$ and $b:V\to V'$ is defined with Koszul's rule,
i.e$.$   we set $(a\otimes b)(u\otimes v) := (-1)^{\vert b \vert \vert u \vert} a(u) \otimes b(v)$ for any homogeneous elements $u\in U,v\in V$.
In this paper, we construct and study  two  functors from $\Cob_G$ to $\PMod_{R,\pm G}$ for some specific rings $R$ and specific subgroups $G\subset R^\times$.

   Our first functor is based on the ``Alexander function'' introduced by Lescop \cite{Lescop}.
For any  compact orientable $3$-manifold $M$ with boundary, this function is defined on an exterior power of the Alexander module of $M$ relative to a boundary point,
and it  takes values in a ring of Laurent polynomials. Lescop's definition proceeds in a rather elementary way using a presentation of the Alexander module.   \\[-0.2cm]

\noindent
\textbf{Theorem I.}
{\it  Let $G$ be a finitely generated free abelian group, and let $\Z[G]$ be its group ring. Then there is a   degree-preserving   monoidal functor\\[-0.3cm] 
$$
\Alex :=  \Alex_{G} :  \Cob_G \longrightarrow \PMod_{\Z[G],\pm G}
$$
which, at the level of objects, assigns to any  $(g,\varphi)$ 
the exterior algebra of the  $\varphi$-twisted relative homology group of the pair $(F_g,\star)$.}\\[-0.1cm]

\noindent
The $\Z[G]$-linear map $\Alex(M,\varphi)$ associated to  a morphism $(M,\varphi)$ of $\Cob_G$
is defined in a very simple way from the Alexander function of $M$ using the decomposition of $\partial M$ into two parts, $\partial_- M$ and $\partial_+ M$.
The fact that the Alexander function gives rise to a functor on the category of cobordisms is somehow implicit in \cite{Lescop},
where  Lescop studies the behaviour of her invariant under some specific gluing operations.
  As it contains the Alexander polynomial of knots in a natural way, we call $\Alex$ the \emph{Alexander functor}.

Since the  works of Milnor \cite{Milnor_duality} and Turaev \cite{Turaev_Alexander}, it is known that the Alexander polynomial of knots 
and $3$-manifolds can be interpreted as a special kind of abelian Reidemeister torsion.
We follow this direction to define our second functor, which we call the \emph{Reidemeister functor}.
In the sequel, the category $\PMod_{R,\pm G}$ associated to a field $R:=\F$ and a subgroup $G $ of $\F^\times = \F \setminus \{0\}$ is denoted by $\PVect_{\F,\pm G}$.\\[-0.2cm]

\noindent
\textbf{Theorem II.}
{\it  Let $\F$ be a field and let $G$ be a  subgroup of $\F^\times$. Then there  is a   degree-preserving   monoidal functor\\[-0.3cm] 
$$
\Reid :=\Reid_{\F, G} :  \Cob_{G} \longrightarrow \PVect_{\F,\pm G}
$$
which, at the level of objects, assigns to any  $(g,\varphi)$ 
the exterior algebra of the $\varphi$-twisted relative homology group of the pair $(F_g,\star)$.}\\[-0.1cm]

\noindent
The construction of the functor $\Reid$ uses the elementary theory of Reidemeister  torsions,
but note that we need to consider cell chain complexes which are \emph{not} necessarily  acyclic.  
When $G$ is a finitely generated free abelian group and $\F:=Q(G)$ is the field of fractions of $\Z[G]$, we recover the functor $\Alex$ by extension of scalars. 
Thus it suffices to study the functor $\Reid$ and this is done using  basic properties of combinatorial torsions.
For instance, we compute its restriction to the monoid of homology cobordisms (which includes the mapping class group of a surface):
we find that the representation induced by $\Reid$ is  equivalent to the Magnus representation 
combined with the Reidemeister torsion of  cobordisms relative to the top surface. 
We also give a formula for $\Reid$ in terms of Heegaard splittings 
and we show that $\Reid$ satisfies some duality properties, which generalize the symmetry properties of the  Alexander polynomial of knots and $3$-manifolds.

It is expected that Turaev's refinements of the Reidemeister torsion \cite{Turaev_RTKT,Turaev_Euler} can be adapted to refine $\Reid$ to a kind of ``monoidal'' degree-preserving functor 
from $\Cob_G$ to the category $\mathsf{grVect}_{\F}$ of graded $\F$-vector spaces:   the sign ambiguity would presumably be fixed using homological orientations on the manifolds, 
while the ambiguity in $G$ would be fixed by adding Euler structures.
(Observe however that, since we use  Koszul's rule and we allow morphisms in $\mathsf{grVect}_{\F}$ to have non-zero degree, 
this category is not monoidal in the usual sense of the word.)

We now explain how our constructions are related to prior works. Soon after the emergence of quantum invariants of $3$-manifolds in the late eighties, 
there have been several works which showed how to interpret the classical Alexander polynomial  in this new framework.
A more general problem was then to extend the Alexander polynomial to a functor from a category of cobordisms to a category of vector spaces
following, as close as possible, the axioms of a TQFT  \cite{Atiyah}. This problem has been solved by Frohman and Nicas who used 
elementary intersection theory in $\hbox{U}(1)$-representation varieties of surfaces \cite{FN_U(1)}. 
(See also \cite{FN_PU(N)} for a much more general  construction using $\hbox{PU}(N)$-representations.)
Later, Kerler showed that the Frohman--Nicas functor is in fact equivalent to a TQFT based on a certain  quasitriangular Hopf algebra \cite{Kerler_HQFT}.
The Alexander polynomial of a knot $K$ in an integral homology $3$-sphere $N$ is recovered from this functor
 by taking the ``graded'' trace of the endomorphism associated  to the cobordism that one obtains by ``cutting'' $N \setminus K$ along a Seifert surface of $K$. 
It turns out that, in the case $G=\{1\}$,   the Alexander functor $\Alex$ is equivalent to the Frohman--Nicas functor.
Note that the way how their functor determines the Alexander polynomial is somehow \emph{extrinsic}, in that it goes through  the choice of a Seifert surface.
On the contrary, the  functor $\Alex$ for $G=\Z$   \emph{intrinsically} contains   the Alexander polynomial of oriented  knots in oriented integral homology $3$-spheres
by considering any  knot of this type as a ``bottom knot'' in the style of \cite{Habiro}, i.e$.$ by regarding  its exterior as a morphism $1\to 0$ in $\Cob_G$.
Since this functorial extension of the Alexander polynomial applies to cobordisms $M$ equipped with an element of $H^1(M;\Z)$,
it should be regarded as a kind of HQFT with target  $\operatorname{K}(\Z,1)$ -- see \cite{Turaev_HQFT} -- rather than a TQFT.

Our constructions are also related  to the work of Bigelow, Cattabriga and the first author  \cite{BCF}, 
which provides a functorial extension of the Alexander polynomial to the category of tangles {instead of} the category of cobordisms.
To describe this relation, let $\TCob$ be the  monoidal category  whose objects are pairs of non-negative integers $(g,n)$ --
corresponding to surfaces $F_g$ with $n$ punctures -- and whose morphisms are cobordisms with tangles inside. 
Clearly the  category $\TCob$ contains  the category $\Cob$ of \cite{CY}  as well as the usual  category $\Tang$ of (unoriented) tangles  in the standard ball;
for any abelian group $G$, there is an obvious refinement $\TCob_G$ of the category $\TCob$.
When $G$ is the infinite cyclic group generated by~$t$, the usual category $\Tang_+$ of oriented tangles in the standard ball can be regarded as a subcategory of $\TCob_G$
by only considering those representations of tangle exteriors that send any oriented meridian to the generator~$t$. 
The functors $\Alex$ and $\Reid$  constructed in this paper  
could be extended to the category $\TCob_G$ using similar methods,  but with more technicality.  
When $G$ is infinite cyclic, the restriction  of the resulting functor $\Alex:\TCob_G \to \PMod_{\Z[G],\pm G}$  to  $\Tang_+$
would coincide with the ``Alexander representation of tangles'' constructed in  \cite{BCF}.
We also mention Archibald's extension of the Alexander polynomial \cite{Archibald}, which is based on diagrammatic presentations of tangles:
her invariant seems to be very close to the  invariant constructed in \cite{BCF} and it is stronger since it is defined \emph{without} ambiguity in $\pm G$.

Finally, our approach is  related to the work of Cimasoni and Turaev on ``Lagrangian representations of tangles'' \cite{CT1,CT2}.
These representations are functors from the category $\Tang_+$ to the category of ``Lagrangian relations'' 
(which generalizes the category of $\Z[t^{\pm 1}]$-modules equipped with non-degenerate skew-hermitian forms)
and, for string links, they are equivalent to the (reduced) Burau representation \cite{LD,KLW}.
The constructions of \cite{CT1,CT2} could be adapted to the case of cobordisms in order to obtain a functor 
from  $\Cob_G$ to the category of ``Lagrangian relations'' over the ring~$\Z[G]$.
In the case of homology cobordisms,  the resulting functor   would be  equivalent to the (reduced) Magnus representation 
but it would miss the relative Reidemeister torsion: so it would be weaker than the functor $\Alex$.

The paper is organized as follows. A first part deals exclusively with the Alexander functor: 
\S \ref{sec:Alexander_functor} gives the construction of the functor $\Alex$ (Theorem I)
and \S \ref{sec:A_knots} explains how  the classical Alexander polynomial of knots is contained in $\Alex$.
Next, the Reidemeister functor is constructed in \S \ref{sec:Reidemeister_functor} (Theorem II) and it is proved to be a generalization of  $\Alex$ in \S \ref{sec:back_to_A}.
(Thus, we provide two different proofs of the functoriality of $\Alex$.)
Starting from there, we focus  on the study of  $\Reid$ and indicate the resulting properties for $\Alex$.
The abelian Reidemeister torsions  of knot exteriors and closed $3$-manifolds are shown to be determined by $\Reid$ in \S \ref{sec:R_knots}.
The functor~$\Reid$ restricts to a projective representation of the monoid of homology cobordisms, which we fully compute  in \S \ref{sec:homology_cobordisms}.
We also explain in \S \ref{sec:Heegaard} how to calculate $\Reid$ using Heegaard splittings of cobordisms, 
and we prove in \S \ref{sec:duality} a duality result for $\Reid$ which involves the twisted intersection form of surfaces.
Finally, the paper ends with a short appendix recalling the definition and basic properties of the torsion of chain complexes.\\
 
\noindent
\textbf{Notation and conventions.}
Let $R$ be a commutative ring. The exterior algebra of an $R$-module $N$ is denoted by 
$$
\Lambda N= \bigoplus_{i\geq 0} \Lambda^i N \quad \hbox{where } \Lambda^0 N = R ;
$$
 the multivector $v_1 \wedge \cdots \wedge v_i \in \Lambda^i N$ defined by a finite family  $v=(v_1,\dots,v_i)$ of elements of  $N$ is still denoted by $v$. 
 If $N$ is free of rank $n$,  a \emph{volume form} on $N$ is an isomorphism of $R$-modules  $\Lambda^n N \to R$.

Let $X$ be a topological space with base point $\star$.
The maximal abelian cover of $X$  based at $\star$ 
 is denoted by   $p_X:\widehat{X}\to X$, and the preferred lift of $\star$ is denoted by $\widehat{\star}$.
 (Here we assume  the appropriate assumptions on $X$ to have a universal cover.)
For any oriented loop $\alpha$ in $X$ based at $\star$, 
the unique lift of $\alpha$ to $\widehat X$ starting at $\widehat \star$ is denoted by $\widehat \alpha$.

Unless otherwise specified, (co)homology groups are taken with coefficients in the ring of integers $\Z$;
(co)homology classes are denoted with square brackets $[-]$.
For any subspace $Y \subset X$ such that $\star \in Y$ and any ring homomorphism $\varphi : \Z[H_1(X)] \to R$,  
we denote by $H^{\varphi}(X,Y)$ the \emph{$\varphi$-twisted homology} of the pair $(X,Y)$, namely
$$
H^{\varphi}(X,Y) = H(C^\varphi(X,Y))
\quad  \hbox{where} \ C^\varphi(X,Y) :=  R \otimes_{\Z[H_1(X)]} C\big(\widehat{X},p_X^{-1}(Y)\big).
$$
If $(X',Y')$ is another  pair of spaces  and $f:(X',Y') \to (X,Y)$ is a continuous map, 
the corresponding homomorphism $H(X') \to H(X)$ is still denoted by $f$.
If a base point $\star'\in Y'$ is given and $f(\star')=\star$, 
the $R$-linear map $H^{\varphi f}(X',Y') \to H^{\varphi }(X,Y)$ induced by $f$ is also denoted by $f$.\\

\noindent
\textbf{Acknowledgements.} This work was partially supported by the French ANR research project ``Interlow'' (ANR-09-JCJC-0097-01). 
The authors would like to thank the referee for some useful comments.

\section{The Alexander functor $\Alex$} \label{sec:Alexander_functor}

We firstly review  the Alexander function of a  $3$-manifold with boundary following~\cite{Lescop}.
(Note that the terminology ``Alexander function'' has a very different meaning in \cite{Turaev_RTKT}.)
Next, we construct the Alexander functor $\Alex$. In this section, we fix a finitely generated free abelian group $G$;
the extension of a group homomorphism $\varphi:A \to G$ to a ring homomorphism $\Z[A]\to \Z[G]$ is still denoted by $\varphi$.

\subsection{The Alexander function} \label{subsec:Alexander_function}

Let $M$ be a compact connected orientable $3$-manifold with  connected boundary.
We fix a base point $\star \in \partial M$ and a group homomorphism $\varphi : H_1(M) \rightarrow G$.
The \emph{genus} of $M$ is the integer $g(M) := 1 - \chi(M)$, i.e$.$ the genus of the surface $\partial M$.

\begin{lemma} \label{lem:pres}
There exists a presentation of the $\Z[G]$-module $H_1^\varphi(M,\star)$ whose deficiency is $g(M)$.
\end{lemma}

\begin{proof}
We consider a decomposition of $M$ with a single $0$-handle, $s$ $1$-handles and $r$ $2$-handles.
Since the boundary of $M$ has genus $g(M)$, we have $s-r=g(M)$. This handle decomposition defines
a $2$-dimensional complex $X\subset M$ onto which $M$ deformation retracts. The complex $X$
has a single $0$-cell (which we assume to be  $\star$), $s$ $1$-cells and $r$ $2$-cells. 
Thus we obtain a presentation of the $\Z[G]$-module $H_1^\varphi(M,\star)\simeq H_1^\varphi(X,\star)$ with $s$ generators and $r$ relations.
\end{proof}

We now simplify our notation by setting $g:=g(M)$ and $H:= H_1^\varphi(M,\star)$.

\begin{definition}[Lescop \cite{Lescop}]
Consider a presentation of the $\Z[G]$-module $H$ with deficiency $g$:
\begin{equation}\label{eq:presentation_H}
H = \langle \gamma_1,\dots, \gamma_{g+r}\, \vert\, \rho_1,\dots,\rho_r \rangle.
\end{equation}
Let $\Gamma$ be the $\Z[G]$-module freely generated by the symbols $\gamma_1,\dots, \gamma_{g+r}$,
and regard $\rho_1,\dots,\rho_r$  as elements of $\Gamma$.
Then the  \emph{Alexander function}  of $M$ with coefficients $\varphi$ is the $\Z[G]$-linear map
$
\mathcal{A}_M^\varphi: \Lambda^g H \to \Z[G]
$
defined by 
$$ 
\mathcal{A}_M^\varphi(u_1\wedge \cdots \wedge u_g) \cdot \gamma_1\wedge \cdots \wedge  \gamma_{g+r}= 
\rho_1\wedge \cdots \wedge \rho_r \wedge \widetilde{u_1} \wedge \cdots \wedge \widetilde{u_g} \ \in \Lambda^{g+r} \Gamma
$$
for any  $u_1,\dots, u_g \in H$,  which we lift to some $\widetilde{u_1}, \dots, \widetilde{u_g} \in \Gamma$ in an arbitrary way.
\end{definition}

The map $\mathcal{A}_M^\varphi$  can be concretely computed as follows:
 if one considers the $r\times (g+r)$ matrix defined by the presentation \eqref{eq:presentation_H} of $H$,
and if one adjoins to this matrix some row vectors giving $u_1,\dots,u_g$ in the generators $\gamma_1,\dots,\gamma_{g+r}$,
then $\mathcal{A}_M^\varphi(u_1\wedge \cdots \wedge u_g)$ is the determinant of the resulting $(g+r) \times (g+r)$ matrix.
It is shown in \cite[\S 3.1]{Lescop} that, up to multiplication by a unit of $\Z[G]$ (i.e., an element of $\pm G$), 
the map $\mathcal{A}_M^\varphi$ does not depend on the choice of the presentation \eqref{eq:presentation_H}.
  
Let $Q(G)$ be the field of fractions of $\Z[G]$.
The following lemma, which is implicit in \cite{Lescop},
shows that  either the Alexander function is trivial or it induces by extension of scalars a volume form on  $H_Q:=  Q(G) \otimes_{\Z[G]} H$. 

\begin{lemma} \label{lem:nullity}
  We have $\dim H_Q \geq g$, and  $\mathcal{A}_M^\varphi \neq 0$  if and only if $\dim H_Q =  g$. 
\end{lemma}

\begin{proof}
Let $A$ be the $r\times (g+r)$ matrix with entries in $\Z[G]$ corresponding to the presentation \eqref{eq:presentation_H} of the $\Z[G]$-module $H$.
The multiplication $v \mapsto v A$ defines a linear map $Q(G)^{r} \to Q(G)^{g+r}$ whose cokernel  is $H_Q $.
Therefore
$$
 \dim H_Q =  (g+r) - \rank A.
$$
Clearly, we have $\rank A\leq r$ so that $  \dim H_Q \geq g$. 

Assume  that $ \dim H_Q > g$ and let $A'$ be a matrix obtained by adding $g$ arbitrary rows to $A$.
Then $\rank A < r$ so that all the minors of $A$ of order $r$ vanish. 
By expanding the determinant of $A'$ successively along the last $g$ rows, we see that $\det A'=0$
and deduce that $\mathcal{A}_M^\varphi=0$. 

  Assume that $ \dim H_Q  =g$. Then $\rank A=r$ so that $A$ has a non-zero minor $D$ of order $r$.
Let $1\leq i_1 <\cdots <i_g \leq g+r$ be the indices of the columns of $A$ not pertaining to $D$.
Then $\mathcal{A}_M^\varphi(\gamma_{i_1} \wedge \cdots \wedge \gamma_{i_g}) =D \neq 0$. 
\end{proof}

\subsection{Definition of $\Alex$} \label{subsec:Alexander_functor}

In order to define a functor $\Alex$, we associate to any object $(g,\varphi)$ of $\Cob_G$  the exterior algebra
$$
\Alex(g,\varphi) :=  \Lambda\, H_1^\varphi(F_g,\star)
$$
of the $\Z[G]$-module  $H^\varphi(F_g,\star)= H_1^\varphi(F_g,\star)$, which is  free of rank $2g$.
Next, we  associate to any morphism $(M,\varphi)\in \Cob_G\big(({g_-},\varphi_-),({g_+},\varphi_+)\big)$ a $\Z[G]$-linear map 
$$ 
\Alex(M,\varphi) : \Lambda\,  H_1^{\varphi_-}(F_{g_-},\star) \longrightarrow \Lambda\,  H_1^{\varphi_+}(F_{g_+},\star)
$$
of degree  $\delta\!g:=g_+-g_-$ as follows.
We denote by $I$ the interval $m(\star\times [-1,1])$, which connects the base point of the bottom surface $\partial_- M$ to that of the top surface $\partial_+ M$.
We set $H:=H_1^\varphi(M,I)$, $H_\pm:=H_1^{\varphi_\pm}(F_{g_\pm},\star)$ and $g:=g_++g_-$.
Then, for any integer $j\geq 0$, the image $\Alex(M,\varphi)(x)\in \Lambda^{j+ \delta\!g}H_+ $ of any  $x \in \Lambda^j H_-$ is defined by the following property:
$$
\forall y \in \Lambda^{g - j} H_+, \
\mathcal{A}_M^\varphi \left( \Lambda^j m_-(x) \wedge \Lambda^{g-j} m_+(y) \right) =  \vol  \big( \Alex(M,\varphi)(x) \wedge y\big).
$$
Here $\vol:\Lambda^{2g_+} H_+  \to \Z[G]$ is an arbitrary {volume form} on  $H_+$.
Due to the choices of $\vol$ and of the presentation of $H$, the map  $\Alex(M,\varphi)$   is only defined up to multiplication by an element of $\pm G$.
Besides, observe that $A(M,\varphi)$ is trivial on $ \Lambda^j H_-$ for any  $j<\max(0,-\delta\!g)$  and any $j>\min(g,2g_-)$.

The next two lemmas show that the above paragraph
defines a monoidal functor $\Alex$ from $\Cob_G$ to  $\PMod_{\Z[G],\pm G}$, which proves Theorem  I of the Introduction.
The first lemma is related  to  Property 6 of the Alexander function in \cite{Lescop},
  while the second lemma seems to be new.  

\begin{lemma} \label{lem:monoidality}
For any morphisms $(M,\varphi)\in \Cob_G((g_-,\varphi_-),(g_+,\varphi_+))$  and $(N,\psi)\in \Cob_G((h_-,\psi_-),(h_+,\psi_+))$,  we have
\begin{equation}\label{eq:monoidality}
\Alex\big( (M,\varphi) \otimes (N,\psi)\big) = \Alex(M,\varphi) \otimes \Alex(N,\psi).
\end{equation}
\end{lemma} 

\begin{proof}  
We set $g:=g_++g_-$, $h:=h_++h_-$, $\delta\!g:=g_+-g_-$, $\delta\!h:=h_+-h_-$ and  
$$
H_\pm^M:=H_1^{\varphi_\pm}(F_{g_\pm},\star), \ H_\pm^N:=H_1^{\psi_\pm}(F_{h_\pm},\star), \
H_\pm:=H_1^{\varphi_\pm \oplus \psi_\pm}(F_{g_\pm+ h_\pm},\star),
$$
$$
H^M := H_1^\varphi(M,I), \quad H^N:= H_1^\psi(N,I), \quad H:= H_1^{\varphi \oplus \psi}(M \sharp_\partial N, I).
$$
In the statement of the lemma and in the proof below, we identify
$$
\Alex\big((g_\pm,\varphi_\pm) \otimes (h_\pm,\psi_\pm) \big) = \Alex(g_\pm+h_\pm,\varphi_\pm \oplus \psi_\pm) 
= \Lambda H_\pm =   \Lambda  \big( H_\pm^M \oplus H_\pm^N \big) 
$$
in the obvious way with 
$$
\Lambda H_\pm^M \otimes   \Lambda H_\pm^N= \Alex(g_\pm,\varphi_\pm) \otimes \Alex (h_\pm,\psi_\pm).
$$
Since the intersection of $M$ and $N$ in $M\sharp_\partial N$ is a $2$-disk which retracts onto $I$, 
the Mayer--Vietoris theorem gives an isomorphism $ H^M \oplus H^N \stackrel{\simeq} \longrightarrow H$.
If $\rank H^M>g$, then $\mathcal{A}_M^\varphi=0$  by Lemma \ref{lem:nullity} so that $\Alex(M,\varphi)=0$; the same lemma applied to $N$  shows that  
$$
\rank H =\rank H^M+\rank H^N >g+h
$$
so that $\Alex\big( (M,\varphi) \otimes (N,\psi)\big)=0$ and \eqref{eq:monoidality} trivially holds true.
Therefore, we can assume in the sequel that $\rank(H^M)=g$ and $\rank(H^N)=h$.

Let $x:= x^M \otimes x^N \in \Lambda^{i}  H_-^M  \otimes \Lambda^j H_-^N \subset \Lambda^{i+j} H_-$:
we aim at showing that $a := \Alex \big( (M,\varphi) \otimes  (N,\psi) \big) (x)$ is equal to 
$$
a' := \big( \Alex (M,\varphi) \otimes  \Alex(N,\psi) \big) (x) = (-1)^{i \delta\! h} \Alex (M,\varphi) (x^M) \otimes \Alex (N,\psi) (x^N).
$$
(Recall that we are using Koszul's rule in the definition of the tensor product of morphisms in the category $\PMod_{\Z[G],\pm G}$.)
It is enough to prove that, for any integers $p,q\geq 0$ such that $p+q= (g+h)-(i+j)$
and any $y := y^M \otimes y^N \in  \Lambda^{p} H_+^M \otimes \Lambda^{q} H_+^N \subset \Lambda^{p+q}H_+$, the identity
\begin{equation}\label{eq:a=a'}
\vol (a\wedge y)= \vol (a'\wedge y) 
\end{equation}
holds true  up to multiplication by an element of $\pm G$ independent of $x,y$ (and, in particular, independent of $i,j,p,q$).
In the sequel, we fix some volume forms $\vol^M$ and $\vol^N$ on $H^M_+$ and $H^N_+$ respectively,
and we assume that the volume form  $\vol$ on $H_+=H^M_+\oplus H^N_+$ is defined by 
\begin{equation} \label{eq:vol_vol}
\vol(u \wedge v) = \vol^M(u) \cdot \vol^N(v) 
\end{equation}
for any $u  \in \Lambda^{2g_+} H_+^M$ and $v \in \Lambda^{2h_+} H_+^N$.
By definition of $\Alex$, we have
\begin{equation}\label{eq:vol(a.y)}
 \vol ( a \wedge y)   = \mathcal{A}_{M \sharp_\partial  N}^{\varphi \oplus \psi} \left( \Lambda^i  m_-(x^M) \wedge  \Lambda^j  n_-(x^N) 
\wedge \Lambda^{p}  m_+(y^M) \wedge \Lambda^{q} n_+(y^N) \right).
\end{equation}
If $p>g-i$, then $i+p> \rank (H^M)$ by our assumptions, 
so that $\Lambda^i  m_-(x^M) \wedge \Lambda^{p}  m_+(y^M) \in \Lambda^{i+p} H^M$ is torsion;
we deduce that $\vol(a\wedge y)=0$; on the other hand,  
the degree of $\Alex (M,\varphi) (x^M)\wedge y^M \in \Lambda H_+^M$ is $i+\delta\!g+ p>2g_+$
so that $\vol(a'\wedge y)=0$ as well; thus \eqref{eq:a=a'} trivially holds true if $p>g-i$.
If $p<g-i$, then $q>h-j$ and the same conclusion applies. Therefore, we can assume  in the sequel that $p=g-i$ and $q=h-j$.

To proceed,  we consider  a presentation  $H^M = \langle \gamma_1, \dots, \gamma_{g+r}\, \vert\, \rho_1, \dots, \rho_r \rangle$
and a presentation $H^N = \langle \mu_1,\dots, \mu_{h+s}\, \vert\, \zeta_1,\dots, \zeta_s \rangle$.
By the above-mentioned isomorphism between $H^M\oplus H^N$ and $H$, we obtain a presentation
$$
H = \langle \gamma_1, \dots, \gamma_{g+r}, \mu_1,\dots, \mu_{h+s}\, \vert\, \rho_1, \dots, \rho_r , \zeta_1,\dots, \zeta_s \rangle.
$$
Note that, with these choices of presentations, the matrix corresponding to $H$ is the direct sum of the matrices corresponding to $H^M$ and $H^N$.
Therefore, we get
\begin{eqnarray*}
\vol( a \wedge y)     
& \stackrel{\eqref{eq:vol(a.y)} }{=}& (-1)^{is+p(s+j)} \mathcal{A}_{M}^\varphi \left( \Lambda^i  m_-(x^M) \wedge  \Lambda^{g -i}  m_+(y^M) \right) \\
&& \qquad \qquad \qquad  \cdot \mathcal{A}_{N}^\psi \big( \Lambda^j n_-(x^N) \wedge  \Lambda^{h -j}  n_+(y^N)  \big) \\
&= &(-1)^{is+p(s+j)}  \vol^M \big( \Alex (M,\varphi) (x^M) \wedge y^M\big) \cdot  \vol^N \big( \Alex (N,\psi) (x^N) \wedge y^N\big)\\ 
& \stackrel{\eqref{eq:vol_vol}}{=} & (-1)^{is+p(s+j)} \vol \big( \Alex (M,\varphi) (x^M) \wedge y^M \wedge   \Alex (N,\psi) (x^N) \wedge y^N\big) \\
&= & (-1)^{is+p(s+j)+p(j+\delta\!h)} \vol \big( \Alex (M,\varphi) (x^M) \wedge   \Alex (N,\psi) (x^N)  \wedge y^M  \wedge y^N\big) \\
& = &  (-1)^{g(s+h)} \vol (a'\wedge y).
\end{eqnarray*}

\up
\end{proof}
 
\begin{lemma} \label{lem:functoriality}
For any morphisms $(M,\varphi)\in \Cob_G((g_-,\varphi_-),(g_+,\varphi_+))$ 
and $(N,\psi)\in \Cob_G((h_-,\psi_-),(h_+,\psi_+))$ such that $(g_+,\varphi_+)= (h_-,\psi_-)$, we have
$$ 
\Alex\big((N,\psi) \circ (M,\varphi)\big) = \Alex(N,\psi) \circ \Alex(M,\varphi).
$$ 
\end{lemma}

The next   subsection   is devoted to the proof of Lemma \ref{lem:functoriality}.

\subsection{Proof of the functoriality of $\Alex$} \label{subsec:functoriality}

We  use the notations of Lemma \ref{lem:functoriality} and we set 
$$
g:=g_-+g_+, \quad h:=h_-+h_+,  \quad f:= g_-+h_+,
$$
$$
\delta\!g := g_+-g_-, \quad \delta\!h := h_+-h_-, \quad \delta\!f := h_+-g_-,
$$
$$
H^M   := H_1^\varphi(M,I), \quad  H^N   := H_1^\psi(N,I), \quad H:= H_1^{\psi + \varphi}(N \circ M,I).
$$
Let   $v=(v_1,\dots, v_{2g_+})$   be a basis of $H_1^{\varphi_+}(F_{g_+},\star)$:
we set $mv_i:= m_+ ( v_i)$ and $nv_i:=n_-(v_i)$ for  all $i=1, \dots, 2g_+$. 
We consider presentations of the following form: 
$$
H^M = \langle mv_1, \dots, mv_{2g_+}, u_1, \dots, u_r\, |\, \zeta_1, \dots, \zeta_{r+\delta\!g}\rangle,
$$
$$ 
H^N = \langle nv_1, \dots, nv_{2 h_-}, w_1,\dots, w_s\, |\, \rho_1, \dots, \rho_{s - \delta\!h} \rangle.
$$
Applying the Mayer--Vietoris theorem to $N\circ M$, 
we obtain that the $\Z[G]$-module $H$ is generated by
\begin{equation}   \label{eq:symbols}
 mv_1, \dots, mv_{2g_+}, nv_1, \dots, nv_{2 h_-}, u_1, \dots, u_r, w_1,\dots, w_s
\end{equation}
subject to the relations \
$
\zeta_1, \dots, \zeta_{r+\delta\!g}, \rho_1, \dots, \rho_{s - \delta\!h}, mv_1 - nv_1, \dots, mv_{2g_+} - nv_{2g_+}.
$

In the sequel, we set  $H_- := H_1^{\varphi_-}(F_{g_-},\star)$ and $H_+ := H_1^{\psi_+}(F_{h_+},\star)$.
Let $  x   \in \Lambda^j H_-$ and $  y   \in \Lambda^{f-j} H_+$:
we wish to compute  
$$
\mathcal{A}_{N \circ M}^{\psi + \varphi} \left(\Lambda^j  m_-(x) \wedge \Lambda^{f-j}  n_+(y)\right)
$$
using the previous presentation of $H$. For this, we do some computations in $\Lambda^k \Gamma$ 
where $k := 4 g_+  + r + s$ and $\Gamma$ denotes the free $\Z[G]$-module generated by the $k$ symbols listed at \eqref{eq:symbols}.
Set $\zeta:= \zeta_1 \wedge \cdots \wedge \zeta_{r+\delta\!g}$, $\rho := \rho_1 \wedge \cdots \wedge \rho_{s-\delta\!h}$. Then, we have
\begin{eqnarray*}
&&\zeta \wedge \rho \wedge (mv_1-nv_1) \wedge \dots \wedge (mv_{2 g_+} - nv_{2 g_+})
\wedge \widetilde{ \Lambda^j m_-(x) }\wedge \widetilde{ \Lambda^{f-j} n_+(y)}\\
&=&\sum_{P } \
(-1)^{\vert P\vert} \varepsilon_P \cdot \zeta \wedge \rho \wedge mv_{P} \wedge nv_{\overline{P}}  \wedge  \widetilde{\Lambda^j  m_-(x)} \wedge \widetilde{\Lambda^{f-j} n_+(y)}\\
&=&\sum_{P } \
(-1)^{\vert P\vert (j+1)} \varepsilon_P \cdot \left(\zeta \wedge mv_{P}  \wedge \widetilde{ \Lambda^j  m_-(x)} \right) \wedge 
\left(\rho \wedge nv_{\overline{P}}  \wedge \widetilde{\Lambda^{f-j} n_+(y)}\right) \ \in \Lambda^k \Gamma.
\end{eqnarray*}
Here the sums are taken over all parts  $P \subset \{1,\dots, 2 g_+ \}$,
$\overline{P}$ denotes the complement of~$P$, $mv_{P}$ is the wedge of the $mv_i$ for $i\in P$,
$nv_{\overline{P}}$ is the wedge of the $nv_i$ for $i \in \overline{P}$ and $\varepsilon_{P}$ is the signature of the permutation $P \overline{P}$ 
(where the elements of $P$ in increasing order are followed by the elements of $\overline{P}$ in increasing order).
A sign  $(-1)^{(s-\delta\!h)(j+\vert P\vert)}$ is missing in the second sum but, 
since the presentation of $H^N$ is arbitrary of deficiency $h$,
we can assume that its number of relations $(s-\delta\!h)$ is even.

In the sequel, we omit the ``tilde'' notation to distinguish elements of $\Lambda H$ from their lifts to $\Lambda \Gamma$.
Note that, in the  above sums, the multivector $\zeta \wedge mv_{P}  \wedge \Lambda^j m_-(x)$ has degree
$(r+\delta\!g )+ \vert P \vert + j$ which is  greater than $2g_++r$ as soon as $\vert P \vert > g -j$;
similarly, the multivector $\rho \wedge nv_{\overline{P}}  \wedge \Lambda^{f-j} n_+(y)$ has degree
$(s-\delta\!h)+ (2g_+-\vert P \vert)+ (f-j)$ 
which is greater than $2h_-+s$ as soon as $\vert P \vert < g-j$; 
since $2g_++r$ and $2h_-+s$ are respectively the numbers of generators of $H^M$ and $H^N$  in the above presentations,
the summand corresponding to $P$ vanishes for $\vert P \vert > g -j$ and for $\vert P \vert < g-j$. 
Therefore the above sums are actually indexed by the subsets  $P \subset \{1,\dots, 2 g_+ \}$ having cardinality $g-j$, and we get
\begin{eqnarray*}
&&\zeta \wedge \rho \wedge (mv_1-nv_1) \wedge \dots \wedge (mv_{2 g_+} - nv_{2 g_+}) \wedge \Lambda^j m_-(x) \wedge \Lambda^{f-j} n_+(y)\\
&=& \sum_{\vert P \vert = g-j} \varepsilon'_P \cdot 
\left(\zeta \wedge mv_{P}  \wedge \Lambda^j m_-(x) \right) \wedge 
\big(\rho \wedge nv_{\overline{P}}  \wedge \Lambda^{f-j} n_+(y)\big)
\end{eqnarray*}
where we have set $\varepsilon'_P := (-1)^{\vert P\vert (j+1)} \varepsilon_P$.
The summand  is here equal to 
\begin{eqnarray*}
&& \varepsilon_P' \cdot \left(\zeta \wedge mv_{P}  \wedge \Lambda^j m_-(x) \right) \wedge 
\big(\rho \wedge nv_{\overline{P}}  \wedge   \Lambda^{f-j} n_+(y)\big) \\
&=& \varepsilon_P' \cdot  \big( \mathcal{A}_{M}^\varphi (mv_P\wedge  \Lambda^j m_-(x))\! \cdot\! (mv \wedge u)\big) \\
&&\qquad \quad \wedge \big( \mathcal{A}_N^\psi(nv_{\overline{P}}\wedge  \Lambda^{f-j} n_+(y))\! \cdot\! (nv \wedge w)\big) \\
&=& \varepsilon_P' \cdot \mathcal{A}_{M}^\varphi \big(mv_P\wedge  \Lambda^j m_-(x)\big) 
 \mathcal{A}_N^\psi\big(nv_{\overline{P}}\wedge  \Lambda^{f-j} n_+(y)\big)  \cdot \left( mv \wedge  nv \wedge u \wedge w\right).
\end{eqnarray*}
We deduce that
\begin{eqnarray*}
&& \mathcal{A}_{N \circ M}^{\psi + \varphi} \big(\Lambda^j m_-(x) \wedge \Lambda^{f-j}  n_+(y)\big) \\ 
&= &\sum_{ \vert P \vert = g-j} \
\varepsilon_P' \cdot  \mathcal{A}_{M}^\varphi \big(mv_P\wedge \Lambda^j m_-(x)\big) \cdot \mathcal{A}_N^\psi \big(nv_{\overline{P}}\wedge \Lambda^{f-j}  n_+(y)\big)\\
&= &  \mathcal{A}_N^\psi\Big( \sum_{ \vert P \vert = g-j} 
\varepsilon_P' \cdot \mathcal{A}_{M}^\varphi\big(mv_P\wedge \Lambda^j m_-(x)\big) \cdot nv_{\overline{P}}\wedge \Lambda^{f-j}  n_+(y)\Big)\\
&= &  \mathcal{A}_N^\psi\Big( \sum_{ \vert P \vert = g-j}
(-1)^{\vert P \vert}  \varepsilon_P \cdot \vol \big(  \Alex(M,\varphi)(x)\wedge v_P \big) \cdot nv_{\overline{P}}\wedge \Lambda^{f-j}  n_+(y)\Big).
\end{eqnarray*}
We can assume that the basis $v$ of $H_1^{\varphi_+}(F_{g_+},\star)$ is compatible with the chosen volume form $\vol$,
in the sense that $\vol (v_1\wedge \cdots \wedge v_{2g_+})= 1$.
Observe that, for all $z\in \Lambda^{j+\delta\!g} H_1^{\varphi_+}(F_{g_+},\star)$, we have the identities
$$
z = \sum_{ \vert P \vert = g-j} \varepsilon_{\overline{P}}  \cdot\vol   (z\wedge v_P)\cdot v_{\overline{P}}
= \sum_{ \vert P \vert = g-j} (-1)^{\vert P \vert}\cdot \varepsilon_{P}  \cdot\vol (z\wedge v_P)\cdot v_{\overline{P}}
$$
where the sums range over all subsets $P \subset \{1,\dots,2g_+\}$ of cardinality $g-j$. Hence
\begin{eqnarray*}
 \mathcal{A}_{N \circ M}^{\psi + \varphi} \big(\Lambda^j m_-(x) \wedge \Lambda^{f-j}  n_+(y)\big) 
 &= &  \mathcal{A}_N^\psi\big( \Lambda^{j+\delta\!g} n_- \Alex(M,\varphi)(x) \wedge \Lambda^{f-j}  n_+(y)\big)\\
&=& \vol \left(\Alex(N,\psi)\big(\Alex(M,\varphi)(x)\big) \wedge y\right).
\end{eqnarray*}
It follows that 
$\vol\left(\Alex\big((N,\psi) \circ (M,\varphi) \big) (x) \wedge y \right)=\vol \left(\Alex(N,\psi)\big(\Alex(M,\varphi)(x)\big) \wedge y\right)$,
which concludes the proof of Lemma \ref{lem:functoriality}.

\section{Alexander functor and knots} \label {sec:A_knots}

In this section, we relate the functor  $\Alex$ to the classical Alexander polynomial of knots.
We fix a finitely generated free abelian group $G$;
the extension of a group homomorphism $\varphi:A \to G$ to a ring homomorphism $\Z[A]\to \Z[G]$ is still denoted by $\varphi$.

\subsection{The Alexander polynomial of a topological pair} \label{subsec:topological_pair}

Given a finitely generated $\Z[G]$-module $N$ and an integer $i\geq 0$,
the \emph{$i$-th Alexander polynomial} of $N$ is the greatest common divisor of all minors of order $n-i$ 
in an $m \times n$ presentation matrix of $N$. This algebraic invariant is denoted by $\Delta_i N\in \Z[G]/\pm G$.

Let $(X,Y)$ be a pair of topological spaces, and assume that they have the homotopy type of finite CW-complexes.
Consider a group homomorphism $\varphi:H_1(X) \to G$.
The \emph{Alexander polynomial} of $(X,Y)$ with coefficients $\varphi$ is 
$$
\Delta^\varphi(X,Y) := \Delta_0 H_1^\varphi(X,Y) \in  \Z[G]/\pm G.
$$
If $Y$ is empty, we set $\Delta^\varphi(X):= \Delta_0 H_1^\varphi(X)$.

\subsection{The Alexander function in genus one}

Let $M$ be a compact connected orientable $3$-manifold with connected boundary, and fix a base point $\star \in \partial M$. 
Let also $\varphi:H_1(M)\to G$ be a group homomorphism.
The next lemma  generalizes Property 1 of the Alexander function given in \cite{Lescop}.

\begin{lemma}\label{lem:def_one}
Assume that $g(M)=1$ and that $\varphi$ is not trivial.
Then, for any $h\in H:= H_1^\varphi(M,\star)$, we have
$$
\mathcal{A}_M^\varphi(h) = \left\{\begin{array}{ll}
\Delta^\varphi(M) \cdot \partial_*(h) & \hbox{ if } \rank \varphi(H_1(M)) \geq 2,\\
{\displaystyle \Delta^\varphi(M) \cdot \frac{\partial_*(h)}{t-1}} & \hbox{ if } \rank \varphi(H_1(M)) =1 \hbox{ and $t$ is a generator.}
\end{array}\right.
$$
Here $\partial_*:H\to \Z[G]$ is the connecting homomorphism
$H_1^\varphi(M,\star) \to H_0^\varphi(\star)$ in the long exact sequence of the pair $(M,\star)$, 
followed by the canonical isomorphism $H_0^\varphi(\star)\simeq \Z[G]$.
\end{lemma}

We shall deduce Lemma \ref{lem:def_one} from the following. 
 
\begin{lemma}\label{lem:abs_rel}
If $\varphi$ is not trivial, then  $\Delta^\varphi(M) =\Delta_1 H_1^\varphi(M,\star)$.
\end{lemma}

\begin{proof}
The long exact sequence in $\varphi$-twisted homology for the pair $(M,\star)$ gives
$$
0 \longrightarrow H_1^\varphi(M) \longrightarrow H_1^\varphi(M,\star) \longrightarrow H_0^\varphi(\star) \longrightarrow H_0^\varphi(M) \longrightarrow 0.
$$
Since the $\Z[G]$-module $H_0^\varphi(\star)\simeq \Z[G]$ is torsion-free, we deduce that 
\begin{equation} \label{eq:Tors_Tors}
\Tors H_1^\varphi(M) \simeq \Tors H_1^\varphi(M,\star).
\end{equation}
Besides, the above exact sequence implies that
$$
\rank H_1^\varphi(M) - \rank H_1^\varphi(M,\star) + 1 - \rank H_0^\varphi(M) = 0.
$$
We now show that  $\rank H_0^\varphi(M) = 0$.
By considering a cell decomposition of $M$ with $\star$ as a single $0$-cell
and some $1$-cells $e_1,\dots, e_r$, we see that
$$
H_0^\varphi(M) \simeq \Z[G] \big/ \big\langle (g_1-1), \dots, (g_r-1) \big\rangle_{\operatorname{ideal}}
$$ 
where $g_i:= \varphi([e_i])\in G$. Thus we have the short exact sequence of modules
$$
0 \longrightarrow I_\varphi \longrightarrow \Z[G] \longrightarrow H_0^\varphi(M) \longrightarrow 0,
$$
where $I_\varphi$ is the ideal generated by the $\varphi(h)-1$ for all $h\in H_1(M)$.
By tensoring with the field of fractions $Q(G)$, we obtain
$$
0 \longrightarrow  Q(G) \otimes_{\Z[G]} I_\varphi  \longrightarrow Q(G) \longrightarrow Q(G) \otimes_{\Z[G]} H_0^\varphi(M)  \longrightarrow 0.
$$
Since $\varphi$ is not trivial,   $Q(G) \otimes_{\Z[G]} I_\varphi \neq 0$ so that $Q(G) \otimes_{\Z[G]} H_0^\varphi(M) =0$. Hence
\begin{equation}\label{eq:rrr}
\rank H_1^\varphi(M,\star)  = \rank H_1^\varphi(M) + 1.
\end{equation}
We conclude thanks to \eqref{eq:Tors_Tors} and \eqref{eq:rrr} using the following:

\begin{quote}\textbf{Fact.}  \cite[Lemma 4.10]{Blanchfield}.
{\it Let $N$ be a finitely generated $\Z[G]$-module. Then
$$ 
\Delta_i(N) = \left\{\begin{array}{ll}
0 & \hbox{ if } i < \rank(N)\\
\Delta_{i-\rank{N}}(\Tors N) & \hbox{ if } i \geq  \rank(N).
\end{array}\right.
$$}
\end{quote}

\up
\end{proof}

\begin{proof}[Proof of Lemma \ref{lem:def_one}]
Observe that, for any oriented  loop $\rho$ in $M$ based at $\star$,   we have $\partial_*([\widehat{\rho}]) = \varphi([\rho])-1$.
Thus, the greatest common divisor of $\partial_*(H)$ is 
$$
\gcd \partial_*(H)= \gcd \big\{\varphi(x)-1\, \vert\, x\in H_1(M)\big\} \in \Z[G]/\pm G.
$$
Since $\varphi$ is assumed to be non-trivial, we deduce that
$$
\gcd \partial_*(H) =
\left\{\begin{array}{ll}
1 & \hbox{ if } \rank \varphi(H_1(M)) \geq 2,\\
t-1 & \hbox{ if } \rank \varphi(H_1(M)) =1 \hbox{ and $t$ is a generator.}
\end{array}\right.
$$
Therefore, we have to prove that
\begin{equation}\label{eq:with_gcd}
\mathcal{A}_M^\varphi(h)= \Delta^\varphi(M) \cdot \frac{\partial_*(h)}{\gcd \partial_* (H)}. 
\end{equation}
For this, we consider a presentation 
$H = \langle \gamma_1,\dots,\gamma_{r+1}\, \vert\,  \rho_1,\dots,\rho_r\rangle$
and let $A$ be the associated $r \times (r+1)$ matrix. 
We have
$$
\forall z_1,\dots,z_{r+1} \in \Z[G], \
\mathcal{A}_M^\varphi(z_1\gamma_1+\cdots+z_{r+1}\gamma_{r+1}) = 
\sum_{i=1}^{r+1} (-1)^{i+r+1} \det(A_{i}) z_i
$$
where $A_i$ is the matrix $A$ with the $i$-th column removed.
Then Lemma \ref{lem:abs_rel} gives 
\begin{equation}\label{eq:D_D_gcd}
\Delta^\varphi(M) =  \Delta_1 H= \gcd \mathcal{A}_M^\varphi(H).
\end{equation}

It follows that $\Delta^\varphi(M)=0$ if and only if $\mathcal{A}_M^\varphi=0$.
In that case \eqref{eq:with_gcd} trivially holds true: thus we assume  in the sequel  that $\mathcal{A}_M^\varphi \neq 0$.
Lemma \ref{lem:nullity}  implies that $\rank H =1$: it follows that any two $Q(G)$-linear maps $ Q(G) \otimes_{\Z[G]} H \to Q(G)$ are linearly dependent.
Since $\mathcal{A}_M^\varphi \neq 0$ and $\partial_* \neq 0$, we deduce that there exist non-zero elements $D,E\in \Z[G]$ such that
\begin{equation}\label{eq:DE}
\forall h\in H, \quad \mathcal{A}_M^\varphi(h) = \frac{D}{E} \cdot \partial_*(h)
\end{equation}
or, equivalently, $D   \partial_*(h)  = E \mathcal{A}_M^\varphi(h)$ for all $h\in H$.
Hence  $D \gcd \partial_*(H)= E\gcd \mathcal{A}_M^\varphi(H)$ and we deduce from \eqref{eq:D_D_gcd} that
\begin{equation}\label{eq:DE_bis}
\frac{D}{E} = \frac{\Delta^\varphi(M)}{ \gcd \partial_*(H)}.
\end{equation}
The identity \eqref{eq:with_gcd} is then deduced from \eqref{eq:DE} and \eqref{eq:DE_bis}.
\end{proof}

\subsection{The functor $\Alex$ on knot exteriors} \label{subsec:knot_exteriors}

Let $K$ be an oriented knot in an oriented  homology $3$-sphere $N$.
The  \emph{Alexander polynomial} of $K$ is classically defined as
$$
\Delta(K) := \Delta^{\varphi_K}(M_K) =\Delta_0\, H_1^{\varphi_K}\!(M_K) \ \in \Z[G]/\pm G
$$ 
where $M_K$ is the complement of an open tubular neighborhood of $K$ in $N$, $G$ is the infinite cyclic group spanned by $t$,
and $\varphi_K:H_1(M_K)\to G$ is the isomorphism mapping an oriented meridian $\mu\subset \partial M_K$ of $K$ to $t$.
Note that $\Delta(K)$ is a Laurent polynomial in the variable $t$, which is defined up to multiplication by a monomial $\pm t^k$ for $k\in \Z$.

We make $M_K$  a morphism $1\to 0$ in the category $\Cob$ by choosing a boundary-parametrization $m: F(1,0) \to\partial M_K$
such that $\mu_- := m^{-1}(\mu)$  is contained in the bottom surface $F_1$ and goes through the base point $\star$. 
Set  $H_- := H_1^{\varphi_K m_- }(F_1,\star)$.
The following proposition  shows that the knot invariants $\Delta(K)$  and $\Alex(M_K,\varphi_K)$ carry the same topological information.
This is deduced  from Lemma \ref{lem:def_one} applied to $M:=M_K$.

\begin{proposition} \label{prop:Alexander_knot}
With the above notation and for any $h\in \Lambda^i H_-$, we have
$$
\Alex(M_K,\varphi_K)(h) = \left\{\begin{array}{ll} \Delta(K) \cdot \partial_*(h)/(t-1) & \hbox{if } i=1,\\
0 & \hbox{otherwise},
\end{array}\right.
$$
where $\partial_*:H_-\to \Z[G]$ is the connecting homomorphism for the pair $(F_1,\star)$.
In particular,  we have $\Delta(K)= \Alex(M_K,\varphi_K)([\widehat \mu_-])$.
\end{proposition}

\section{The Reidemeister functor $\Reid$} \label{sec:Reidemeister_functor}

In this section, we construct the Reidemeister functor $\Reid$.
We fix a field $\F$ and a subgroup $G$ of $\F^\times$.
  In this section, the extension of  a group homomorphism $\varphi:A \to G$ to a ring homomorphism $\Z[A]\to \F$  is still denoted by $\varphi$.

\subsection{The Reidemeister function} \label{subsec:R_function}

We use the elementary theory of  abelian Reidemeister torsions to construct an analogue of the Alexander function considered in \S \ref{subsec:Alexander_function}. 
Let $M$ be a compact connected orientable $3$-manifold with connected boundary, and let $\varphi:H_1(M) \to G$ be a group homomorphism.
We fix a base point $\star \in \partial M$ and we set $g:=g(M)=1-\chi(M)$.

\begin{lemma} \label{lem:homology}
We have $H_i^{\varphi}(M,\star) = 0$ if $i=0$ or $i>2$. Moreover, we have
$$
\dim H_1^{\varphi}(M,\star) = g + \dim H_2^{\varphi}(M,\star).
$$
\end{lemma}

\begin{proof}
Since $\partial M$ is non-empty, $M$  deformation retracts to a connected $2$-dimensional complex whose only $0$-cell is $\star$: the first assertion follows.
Moreover, we have 
$$
-g= \chi(M) - 1 = \chi(M,\star) = - \dim H_1^{\varphi}(M,\star) + \dim H_2^{\varphi}(M,\star).
$$

\up
\end{proof}

Denote $  H :=   H_1^{\varphi}(M,\star)$  and assume in this paragraph that $\dim H =g$.
We choose a cell decomposition of $M$ where $\star$ is a $0$-cell:
by Lemma \ref{lem:homology}, the homology of the $\varphi$-twisted cell chain complex $C^\varphi(M,\star)$ is concentrated in degree $1$.
For every dimension $i\in\{0,\dots,3\}$, let $n_i\geq 0$ be the number of relative $i$-cells of $(M,\star)$
and order them $\sigma_1^{(i)},\dots,\sigma_{n_i}^{(i)}$ in an arbitrary way. 
For every cell $\sigma$ of $(M,\star)$, we also choose an orientation of $\sigma$ 
and a lift $\hat \sigma$ of $\sigma$ to the maximal abelian cover $\widehat M$ of $M$.
Thus, we get a basis $c:=(c_3,c_2,c_1,c_0)$  of the $\F$-chain complex  $C^{\varphi}(M,\star)$
where, for every  $i\in\{0,\dots,3\}$, the basis of the $\F$-vector space $C^{\varphi}_i(M,\star)$ is given by
$
c_i:= \big(1\otimes \hat\sigma_1^{(i)},\dots,1 \otimes \hat\sigma_{n_i}^{(i)}\big).
$
Then  we  consider the function $H^g \to \mathbb{F}$ defined by
\begin{equation} \label{eq:linear_independent_case}
 (h_1,  \dots , h_g)  \longmapsto  \left\{ \begin{array}{ll}
\tau\big(C^{\varphi}(M,\star); c, (h_1,\dots,h_g) \big) & \hbox{if $h_1Ê\wedge  \cdots  \wedge h_g \neq 0$, } \\
0 & \hbox{otherwise.}
\end{array}\right.
\end{equation}
Here $\tau\left(C; c,h \right)$ denotes the torsion of the finite $\F$-chain complex $C$ with basis $c$ and homological basis $h$: see \S \ref{subsec:torsion_def}.
It follows from the definition of the torsion that the map \eqref{eq:linear_independent_case} is multilinear and alternate: see Lemma \ref{lem:torsion_as_function}.

\begin{definition} \label{def:Reidemeister_function}
The \emph{Reidemeister function} of $M$ with coefficients $\varphi$ is the $\F$-linear map
$\mathcal{R}_M^\varphi:\Lambda^g H \to \F$ defined by \eqref{eq:linear_independent_case} if $\dim H=g$
and by $\mathcal{R}_M^\varphi:= 0$  if $\dim HÊ\neq g$.
\end{definition}

Because of the choice of the orders, orientations, and lifts of the  cells of $(M,\star)$,
the map $\mathcal{R}_M^\varphi$ is only defined up to multiplication by an element of $\pm G\subset \F$.
It remains to justify that $\mathcal{R}_M^\varphi\in \Hom(\Lambda^g H,\F)/\!\pm G$ defines a topological invariant of $M$ 
(i.e$.$, it does not depend on the choice of the cell decomposition).
Note that we do not need  Chapman's result on the topological invariance of the torsion of CW-complexes \cite{Chapman,Cohen} since we are considering here manifolds of dimension $3$.
Specifically, using Whitehead's theory of smooth triangulations and the fact  that the Reidemeister torsion of CW-complexes is invariant under cellular subdivisions, 
we obtain that the above definition of $\mathcal{R}_M^\varphi$ applied to a smooth triangulation of $(M,\star)$ produces an invariant of smooth $3$-manifolds. 
(See \cite[\S 9]{Milnor_Whitehead}  or \cite[\S 3]{Turaev_Euler} for similar arguments which are valid in any dimension.)
Next, we appeal to the $3$-dimensional   Hauptvermutung  to conclude that $\mathcal{R}_M^\varphi$ is an invariant of topological $3$-manifolds.
Thus, we can consider  in Definition \ref{def:Reidemeister_function}  an arbitrary cell decomposition of $(M,\star)$
provided it can be subdivided to a smooth triangulation of $M$.

\subsection{Definition of $\Reid$} \label{subsec:Reidemeister_functor}

The definition of the functor $\Reid$ from the Reidemeister function $\mathcal{R}$ 
goes parallel to the definition of   $\Alex$ from $\mathcal{A}$ (see \S \ref{subsec:Alexander_functor}).
Thus we associate to any object $(g,\varphi)$ of $\Cob_G$  the exterior algebra 
$$
\Reid(g,\varphi) :=  \Lambda\, H_1^{\varphi}(F_g,\star)
$$
of the $\F$-vector space $H^{\varphi}(F_g,\star)=H_1^{\varphi}(F_g,\star)$, which has dimension $2g$.
Next, we  associate to any morphism $(M,\varphi)$ from $({g_-},\varphi_-)$ to $({g_+},\varphi_+)$ an $\F$-linear map 
$$ 
\Reid(M,\varphi) : \Lambda\,  H_1^{\varphi_-}(F_{g_-},\star) \longrightarrow \Lambda\,  H_1^{\varphi_+}(F_{g_+},\star)
$$
of degree  $\delta\!g:=g_+-g_-$ in the following way. 
We set $H :=H_1^{\varphi}(M,I)$ where  $I:=m(\star\times [-1,1])$,  $  H_\pm:=   H_1^{\varphi_\pm}(F_{g_\pm},\star)$ and $g:=g_++g_-$.
Then, for any integer $j\geq 0$,
the image $\Reid(M,\varphi)(x)\in \Lambda^{j+ \delta\!g}H_+$ of any  $x \in \Lambda^j H_-$  is defined by the following property:
$$
\forall y \in \Lambda^{g - j} H_+, \
\mathcal{R}_M^\varphi \left( \Lambda^j m_-(x) \wedge \Lambda^{g-j} m_+(y) \right) = \vol  \big( \Reid(M,\varphi)(x) \wedge y\big).
$$ 
Here $\vol   :\Lambda^{2g_+} H_+ \to \F$ is an arbitrary volume form which is   \emph{integral} in the following sense:
regarding  $H_+$ as  $\F \otimes_{\Z[H_1(F_{g_+})]} H_1(F_{g_+},\star;\Z[H_1(F_{g_+})])$,
 we assume that $\vol$ arises from an arbitrary volume form on the free $\Z[H_1(F_{g_+})]$-module $H_1(F_{g_+},\star;\Z[H_1(F_{g_+})])$. 
Due to the choices of this volume  form and of the ordered/oriented lifts of the cells to $\widehat{M}$,
the map $\Reid(M,\varphi)$  is only defined up to multiplication by an element of $\pm G\subset \F$.
Besides,  $R(M,\varphi)$ is trivial on $ \Lambda^j H_-$ for any  $j<\max(0,-\delta\!g)$  and any $j>\min(g,2g_-)$.  

The next two lemmas show that the above paragraph
defines a monoidal functor $\Reid:\Cob_G \to \PVect_{\F,\pm G}$, which proves Theorem  II of the Introduction.

\begin{lemma} \label{lem:monoidality_for_Reid}
For any morphisms $(M,\varphi)\in \Cob_G((g_-,\varphi_-),(g_+,\varphi_+))$  and $(N,\psi)\in \Cob_G((h_-,\psi_-),(h_+,\psi_+))$,  we have
\begin{equation} \label{eq:monoidality_for_Reid}
\Reid\big( (M,\varphi) \otimes (N,\psi)\big) = \Reid(M,\varphi) \otimes \Reid(N,\psi).
\end{equation}
\end{lemma}

\begin{proof}
We set $g:=g_++g_-$, $h:=h_++h_-$, $\delta\!g:=g_+-g_-$, $\delta\!h:=h_+-h_-$ and  
$$
H_\pm^M:=H_1^{\varphi_\pm}(F_{g_\pm},\star), \ H_\pm^N:=H_1^{\psi_\pm}(F_{h_\pm},\star), \
H_\pm:=H_1^{\varphi_\pm \oplus \psi_\pm}(F_{g_\pm+ h_\pm},\star),
$$
$$
H^M := H_1^{\varphi}(M,I), \quad H^N:= H_1^{\psi}(N,I), \quad H:= H_1^{\varphi \oplus \psi}(M \sharp_\partial N, I).
$$
Since $M$ and $N$ intersect in $M\sharp_\partial N$ along a $2$-disk which retracts onto $I$, 
the Mayer--Vietoris theorem gives an isomorphism $ H^M \oplus H^N \stackrel{\simeq} \longrightarrow H$.
If $\dim(H^M)>g$, then $\mathcal{R}^\varphi_M=0$ by definition, so that $\Reid(M,\varphi)=0$; moreover, 
$$
\dim(H) = \dim(H^M) +  \dim(H^N) > g+h
$$
so that $\Reid\big( (M,\varphi) \otimes (N,\psi)\big)=0$ as well, and \eqref{eq:monoidality_for_Reid} trivially holds true in that case.
Therefore, we can assume that $\dim(H^M)=g$ and  $\dim(H^N)=h$.

Let $x^M=(x_1^M,\dots,x_i^M)$ be a family of  vectors in $H_-^M$ and let $x^N=(x_1^N,\dots,x_j^N)$ be a family of  vectors in $H_-^N$.
We consider the element 
$$
x:= x^M \otimes x^N \in \Lambda^{i}  H_-^M  \otimes \Lambda^j H_-^N \subset \Lambda^{i+j} \big(H_-^M \oplus H_-^N\big) = \Lambda^{i+j} H_- .
$$
We aim at showing  that $r := \Reid \big( (M,\varphi) \otimes  (N,\psi) \big) (x)$ is equal to 
$$
r' := \big( \Reid  (M,\varphi) \otimes   \Reid (N,\psi) \big) (x) = (-1)^{i \delta\!h}\cdot  \Reid (M,\varphi) (x^M) \otimes \Reid (N,\psi) (x^N).
$$
It is enough to prove that, for any integers $p,q\geq 0$ such that $p+q= (g+h)-(i+j)$
and for any families  $y^M=(y^M_1,\dots, y^M_p) \subset H_+^M$ and $y^N=(y^N_1,\dots, y^N_q)  \subset H_+^N$, we have
\begin{equation}\label{eq:r=r'}
\vol    (r\wedge y)= \vol  (r'\wedge y)
\end{equation}
where  $y := y^M \otimes y^N \in  \Lambda^{p} H_+^M \otimes \Lambda^{q} H_+^N \subset \Lambda^{p+q}H_+$.
In fact, we only need to prove \eqref{eq:r=r'}  up to multiplication by an element of $\pm G$, provided this factor is  independent of $i,j,p,q,x$ and $y$.

In the sequel, we fix  integral volume forms $\vol^M$ and $\vol^N$ on $H^M_+$ and $H^N_+$ respectively,
and we assume that the volume form  $\vol$ on $H_+=H^M_+\oplus H^N_+$ is defined by 
\begin{equation} \label{eq:vol_vol_bis}
\vol(u \wedge v) = \vol^M(u) \cdot \vol^N(v) 
\end{equation}
for any $u  \in \Lambda^{2g_+} H_+^M,\, v \in \Lambda^{2h_+} H_+^N$. (So $\vol$ is integral too.)
By definition of $\Reid$, we have
\begin{equation}\label{eq:vol(r.y)}
 \vol ( r \wedge y)   = \mathcal{R}_{M \sharp_\partial  N}^{\varphi \oplus \psi} \left( \Lambda^i  m_-(x^M) \wedge  \Lambda^j  n_-(x^N) 
\wedge \Lambda^{p}  m_+(y^M) \wedge \Lambda^{q} n_+(y^N) \right).
\end{equation}
If $p>g-i$, then we have $i+p> \dim (H^M)$ by our assumptions
and we obtain $\Lambda^i  m_-(x^M) \wedge \Lambda^{p}  m_+(y^M)=0 \in \Lambda^{i+p} H^M$;
we deduce that $\vol (r\wedge y)=0$; on the other hand,  
the degree of the multivector $\Reid(M,\varphi) (x^M)\wedge y^M \in \Lambda H_+^M$ is $i+\delta\!g+ p>2g_+$
so that $\vol  (r'\wedge y)=0$ as well; thus \eqref{eq:r=r'} trivially holds true if $p>g-i$.
If $p<g-i$, then $q>h-j$ and the same conclusion applies. Therefore, we can assume that $p=g-i$ and $q=h-j$ in the sequel.

Since $H^M \oplus H^N\simeq H$, $k:=\big(m_- (x^M), m_+(y^M), n_- (x^N), n_+(y^N)\big)$ is a basis of $H$ 
if, and only if, the families $k^M:=\big(m_- (x^M), m_+(y^M)\big)$ and $k^N:=\big(n_- (x^N), n_+(y^N)\big)$ are basis of $H^M$  and  $H^N$ respectively.
If the former condition is not satisfied, then $\vol  ( r \wedge y)$ is zero by \eqref{eq:vol(r.y)} and, if the latter condition is not satisfied, 
then  $\vol ( r' \wedge y)$ is trivial as well since we have 
\begin{eqnarray}
\notag  \vol ( r' \wedge y)  &=& (-1)^{i\delta\!h} \vol \left( \Reid (M,\varphi) (x^M) \wedge \Reid (N,\psi) (x^N) \wedge y^M \wedge y^N \right) \\
\notag &=& (-1)^{i\delta\!h+p(j+\delta\! h)} \vol \left( \Reid (M,\varphi) (x^M) \wedge y^M \wedge \Reid (N,\psi) (x^N)  \wedge y^N \right) \\
\notag &\stackrel{\eqref{eq:vol_vol_bis}}{=}&(-1)^{gh+pj} \vol^M\! \left( \Reid (M,\varphi) (x^M)  \wedge y^M  \right)  \cdot  \vol^N\! \left(  \Reid (N,\psi) (x^N) \wedge y^N \right) 
\end{eqnarray}
or, equivalently,
\begin{eqnarray}
\label{eq:vol(r'.y)}  \vol ( r' \wedge y)   &=& (-1)^{gh+pj} \mathcal{R}_M^\varphi \left( \Lambda^i m_-(x^M) \wedge \Lambda^{g-i} m_+(y^M) \right) \\
\notag && \qquad \qquad  \cdot \mathcal{R}_N^\psi \left( \Lambda^j n_-(x^N) \wedge \Lambda^{h-j} n_+(y^N) \right).
\end{eqnarray}
Therefore, we can assume in the sequel  that  $k$ is a basis of $H$. 

Consider next the twisted cell chain complexes $C:=C^{\varphi \oplus \psi}(M\sharp_\partial N,I)$, 
$C^M:=C^{\varphi }(M,I)$ and $C^N:=C^{\psi }(N,I)$. There is a short exact sequence of $\F$-chain complexes
\begin{equation}\label{eq:ses_DCCC}
\xymatrix{
0 \ar[r] & D \ar[r] & C^M \oplus C^N \ar[r] & C \ar[r] & 0
}
\end{equation}
where $D$ is the (un-)twisted cell chain complex of the disk $M\cap N \subset  M\sharp_\partial N$ relatively to~$I$. Clearly, $D$ is acyclic.
By the multiplicativity property of torsions (see Theorem~\ref{th:multiplicativity} and Example \ref{ex:direct_sum}), we obtain
$$
\varepsilon \cdot \tau(C;c,k) \cdot \tau(D;d) \cdot \tau\left(\mathcal{H};((k^M,k^N),k) \right)\\
 =  \tau\big(C^M;c^M,k^M\big) \cdot \tau\big(C^N;c^N,k^N\big)
 $$
for some appropriate choices  of ordered/oriented lifts of the relative cells,  which result in bases $c,d,c^M,c^N$ of the chain complexes.
Here $\varepsilon$ is a sign not depending on $i,j,p,q,x,y$,
and  $\mathcal{H}$ is the  long exact sequence in homology 
$$
0 \longrightarrow \cdots   \longrightarrow  0  \longrightarrow H^M \oplus H^N  \longrightarrow H  \longrightarrow 0  \longrightarrow 0   \longrightarrow  0
$$ 
induced by \eqref{eq:ses_DCCC}, which we view as a finite acyclic $\F$-chain complex concentrated in degrees $3,4$ and with basis $\big((k^M,k^N),k\big)$.
By definition of $k$, $k^M$ and $k^N$, we have 
$
\tau\big(\mathcal{H};((k^M,k^N),k)\big) =1
$
and, since the intersection disk $M\cap N$ can be reduced to $I$ by elementary collapses, the scalar $T:=\tau(D;d)$ belongs to $\pm G$. We conclude that
\begin{eqnarray*}
\vol ( r \wedge y) &\stackrel{\eqref{eq:vol(r.y)}}{=}& (-1)^{pj} \cdot  \tau\left(C;c,k\right) \\
&=&  (-1)^{pj} \varepsilon T^{-1} \cdot \tau\big(C^M;c^M,k^M\big) \cdot \tau\big(C^N;c^N,k^N\big) \\
& \stackrel{\eqref{eq:vol(r'.y)}}{=} & (-1)^{gh} \varepsilon T^{-1} \cdot \vol   ( r' \wedge y). 
\end{eqnarray*}

\up
\end{proof}  
 
\begin{lemma} \label{lem:functoriality_for_Reid}
For any morphisms $(M,\varphi)\in \Cob_G((g_-,\varphi_-),(g_+,\varphi_+))$ 
and $(N,\psi)\in \Cob_G((h_-,\psi_-),(h_+,\psi_+))$ such that $(g_+,\varphi_+)= (h_-,\psi_-)$, we have
\begin{equation}\label{eq:functoriality_for_Reid}
\Reid\big((N,\psi) \circ (M,\varphi)\big) = \Reid(N,\psi) \circ \Reid(M,\varphi).
\end{equation}
\end{lemma}

  The next subsection is devoted to the proof of Lemma \ref{lem:functoriality_for_Reid}.

\subsection{Proof of the functoriality of $\Reid$}

We  use the notations of Lemma \ref{lem:functoriality_for_Reid} and we set 
$$
g:=g_-+g_+, \quad h:=h_-+h_+,  \quad f:= g_-+h_+,
$$
$$
\delta\!g := g_+-g_-, \quad \delta\!h := h_+-h_-, \quad \delta\!f := h_+-g_-,
$$
$$
H^M:=H_1^{\varphi}(M,I), \quad H^N:=H_1^{\psi}(N,I), \quad H:=H_1^{\psi + \varphi}(N \circ M,I),
$$
$$
K^M:=H_2^{\varphi}(M,I), \quad K^N:=H_2^{\psi}(N,I), \quad K:=H_2^{\psi + \varphi}(N \circ M,I),
$$
$$
H_-:=H_1^{\varphi_-}(F_{g_-},\star), \quad V :=H_1^{\varphi_+}(F_{g_+},\star), \quad  H_+ :=H_1^{\psi_+}(F_{h_+},\star). 
$$
Since $N \circ M$ is obtained from $M$ and $N$ by identifying $\partial_+ M$ to $\partial_- N$,
there is a short exact sequence of  chain complexes
\begin{equation}\label{eq:ses_DCCC'_}
 0 \longrightarrow \underbrace{C^{\varphi_+}(F_{g_+},\star)}_{D:=} \longrightarrow 
 \underbrace{C^{\psi}(N,I)}_{C^N:=} \oplus \underbrace{C^{\varphi}(M,I)}_{C^M:=}
 \longrightarrow \underbrace{C^{\psi + \varphi}(N \circ M,I)}_{C:=} \longrightarrow 0. 
\end{equation}
Let $\mathcal{H}$ be  the corresponding long exact sequence in homology:
$$
 0 \to \cdots \to 0 \to    K^N \oplus K^M\! \to   K \to   V \stackrel{(-n_- , m_+)} \longrightarrow H^N \oplus H^M \!\to  H \to  0 \to  0 \to  0.
$$
If $K^M\neq 0$, then $\dim(H^M) >g$ by Lemma \ref{lem:homology} so that $\mathcal{R}_M^\varphi=0$ and $\Reid(M,\varphi)=0$;
besides, the long exact sequence $\mathcal{H}$ implies that $K\neq 0$ so that $\Reid((N,\psi) \circ (M,\varphi))=0$; 
therefore, \eqref{eq:functoriality_for_Reid} trivially holds true in that case. If $K^N \neq 0$, the same conclusion applies.
So, we can assume  that $K^M=0$ and $K^N=0$ or, equivalently, $\dim H^M=g$ and $\dim H^N=h$.

Let $j\in \{0,\dots,f\}$, and let $x=(x_1 , \dots, x_j)$ and $y=(y_1,\dots,y_{f-j})$ be families  of vectors in $H_-$ and $H_+$ respectively.   
Let $v=(v_1,\dots,v_{2g_+})$ be an arbitrary basis of $V$ and let $\vol^v: \Lambda^{2g_+} V\to \F$ be the volume form such that $\vol^v(v_1\wedge \cdots \wedge v_{2g_+})=1$;
there exists an $\alpha_v\in \F\setminus \{0\}$ such that $\vol = \alpha_v \cdot \vol^v$ is the integral volume form chosen in the definition of the functor $\Reid$. 
We have $\Reid(M,\varphi)(x) \in \Lambda^{j+\delta\!g} V$, hence
$$
\Reid(M,\varphi)(x)  = \sum_{ \vert P\vert =g-j }  \varepsilon_{\overline{P}}  \cdot\vol^v \big(\Reid(M,\varphi)(x) \wedge v_P\big)\cdot v_{\overline{P}}
$$
where the sum is taken over all subsets  $P \subset \{1,\dots, 2 g_+ \}$ of cardinality $g-j$, $\overline{P}$ denotes the complement of $P$, 
$v_{P}$  (respectively $v_{\overline{P}}$) is the wedge of the $v_i$'s for $i\in P$ (respectively $i \in \overline{P}$),
 and $\varepsilon_{\overline P}$ is the signature of the permutation $ \overline{P}P$ 
(where the elements of $\overline P$ in increasing order are followed by the elements of ${P}$ in increasing order).
We deduce that
\begin{eqnarray}
\label{eq:RN_RM}   && \quad \vol   \left( \Reid(N,\psi)\big( \Reid(M,\varphi)(x) \big) \wedge y\right) \\
\notag  &=& \mathcal{R}_N^\psi\left( \Lambda^{j+\delta\!g} n_- \Reid(M,\varphi) (x) \wedge \Lambda^{f-j}  n_+(y)\right) \\
\notag &=& \mathcal{R}_N^\psi\Big( \sum_{ \vert P \vert = g-j}
 \varepsilon_{\overline P} \cdot \vol^v\big(  \Reid(M,\varphi)(x)\wedge v_P \big) \cdot \Lambda^{j+\delta\!g} n_- (v_{\overline{P}})\wedge \Lambda^{f-j} n_+(y)\Big) \\
\notag &=& \alpha_v^{-1} \mathcal{R}_N^\psi\Big( \sum_{ \vert P \vert = g-j}
\varepsilon'_P \cdot \mathcal{R}_M^\varphi\big(\Lambda^{g-j} m_+(v_P) \wedge \Lambda^j m_-(x)  \big) \cdot \Lambda^{j+\delta\! g} n_- (v_{\overline{P}})\wedge \Lambda^{f-j} n_+(y)\Big) \\
\notag &=& \alpha_v^{-1}\sum_{ \vert P \vert = g-j}
\varepsilon'_P \cdot \mathcal{R}_M^\varphi \big(\Lambda^{g-j} m_+(v_P) \wedge \Lambda^j m_-(x)  \big) \cdot  \mathcal{R}_N^\psi\big( \Lambda^{j+\delta\! g} n_- (v_{\overline{P}})\wedge \Lambda^{f-j}  n_+(y)\big)
\end{eqnarray}
where $\varepsilon'_P := \varepsilon_{\overline P} \cdot (-1)^{j(g-j)}$.
If $K\neq 0$, then $\Reid\big((N,\psi)\circ (M,\varphi)\big)=0$; besides,
 the long exact sequence in homology $\mathcal{H}$ shows that there exists a $w\in V\setminus \{0\}$ such that $n_-(w)=0 \in H^N$ and $m_+(w) =0\in H^M$;
since the  basis $v$ of $V$ is arbitrary in \eqref{eq:RN_RM}, we can assume that $v_1=w$. In the last sum indexed by  $P$, 
the vector $w$ appears either in $v_P$ or in $v_{\overline P}$, so that the corresponding summand is always zero;
it follows that $\Reid(N,\psi)\big(\Reid(M,\varphi)(x) \big) \wedge y=0$ for any $x \in \Lambda^j H_-$ and $y \in \Lambda^{f-j} H_+$;
therefore, \eqref{eq:functoriality_for_Reid} trivially holds true in that case. Thus, we can  assume in the sequel that $K=0$ or, equivalently, $\dim H= f$.

It now remains to prove using the above assumptions that, for any families of vectors $x=(x_1 , \dots, x_j)$ in $H_-$ and $y=(y_1,\dots,y_{f-j})$ in $H_+$,
\begin{eqnarray}
\label{eq:RNM}   & &\vol   \left( \Reid\big((N,\psi) \circ (M,\varphi)\big)(x)  \wedge y\right)\\
\notag &=&  \alpha_v^{-1}\sum_{ \vert P \vert = g-j}
\varepsilon'_P \cdot \mathcal{R}_M^\varphi \big(\Lambda^{g-j} m_+(v_P) \wedge \Lambda^j m_-(x)  \big) 
\cdot  \mathcal{R}_N^\psi\big( \Lambda^{j+\delta\! g} n_- (v_{\overline{P}})\wedge \Lambda^{f-j}  n_+(y)\big)
\end{eqnarray}
where, as in the previous paragraph, $v$ is an arbitrary basis of $V$. Assume firstly that $k:=(m_-(x), n_+(y))$ is not a basis of $H$. 
Then 
$$
\Reid\big((N,\psi) \circ (M,\varphi)\big)(x)  \wedge y = \mathcal{R}_{N \circ M}^{\psi + \varphi}\big(\Lambda^j m_-(x) \wedge \Lambda^{f-j} n_+(y)\big)
$$ 
is zero. Besides, the long exact sequence $\mathcal{H}$ implies that there exists $w\in V$ such that 
\begin{eqnarray*}
m_+(w)&=& a_1  m_-(x_1) + \cdots + a_{j} m_-(x_{j}) \in H^M, \\
-n_-(w) &=& b_1  n_+(y_1) + \cdots + b_{f-j} n_+(y_{f-j}) \in H^N
\end{eqnarray*}
where $a_1,\dots, a_j,b_1,\dots,b_{f-j} \in \F$ are not all zeroes.
If $w=0$,  then we have $\Lambda^j m_-(x) =0 \in \Lambda^j H^M$ or $\Lambda^{f-j} n_+(y) =0 \in \Lambda^{f-j} H^N$
(depending on whether we can find a non-zero scalar among the $a_i$'s or among the $b_i$'s); 
in both cases, the second term of \eqref{eq:RNM} is trivial. If $w\neq0$, then we take a basis $v$ of $V$ such that $v_1=w$
and we easily see that the second term of \eqref{eq:RNM} is trivial in that case too. Therefore, we can assume in the sequel that $k=(m_-(x), n_+(y))$ is a basis of $H$.

We  now fix a basis $v=(v_1,\dots,v_{2g_+})$ of $V$ such that  $ \vol (v)=1$ and we prove \eqref{eq:RNM} with $\alpha_v=1$.
Let also $k^M$ and $k^N$ be arbitrary bases of $H^M$ and $H^N$, respectively.
By the multiplicativity property of  torsions  (see Theorem \ref{th:multiplicativity} and Example \ref{ex:direct_sum}), we deduce from  \eqref{eq:ses_DCCC'_}  that
\begin{eqnarray}
\label{eq:many_torsions} & &\tau ( D; d , v) \cdot \tau (C; c,k) \cdot \tau\!\left(\mathcal{H}; \big(v,(k^N,k^M),k\big)\right) \\
\notag & = & \pm  \tau(C^N;c^N,k^N) \cdot  \tau(C^M;c^M,k^M)  \ \in \F
\end{eqnarray}
for some appropriate choices  of ordered/oriented lifts of the relative cells,  which result in bases $d,c,c^M,c^N$ of the chain complexes.
The sign appearing in \eqref{eq:many_torsions} only depends on the dimensions of the complexes $C,D,C^M,C^N$ and the dimensions of their homology groups.
The sequence $\mathcal{H}$ is viewed  here as  a finite acyclic  $\F$-chain complex concentrated in degrees $3,4,5$; its torsion is
\begin{eqnarray*} 
 \tau\left(\mathcal{H}; \big(v,(k^N,k^M),k\big) \right) 
&=&  \left[\frac{\big((-n_-,m_+)(v),\hbox{\scriptsize lift of $k$ to } H^N \oplus H^M  \big)}{(k^N,k^M)}\right]^{-1}\\
&=&  \left[\frac{(k^N,k^M)}{\big((-n_-,m_+)(v),\hbox{\scriptsize lift of $k$ to } H^N \oplus H^M  \big)}\right],
\end{eqnarray*} 
where the symbol $\left[\frac{a}{b}\right]$ stands for the determinant of the square matrix expressing a family of vectors $a$ in the basis $b$ of $H^N \oplus H^M$.
We have $\tau ( D; d , v)\in \pm G$ since $(F_{g_+},\star)$ has the simple homotopy type of a wedge of circles relative to its vertex.
We deduce from \eqref{eq:many_torsions} that
\begin{eqnarray*}
&&\mathcal{R}_{N \circ M}^{\psi + \varphi } \big( \Lambda^j m_-(x) \wedge \Lambda^{f-j} n_+(y) \big)  \cdot  \left[\frac{(k^N,k^M)}{\big((-n_-,m_+)(v),\hbox{\scriptsize lift of $k$ to } H^N \oplus H^M  \big)}\right]\\
&=& \beta_v \cdot \mathcal{R}_M^{\varphi} ( k^M) \cdot  \mathcal{R}_N^{\psi} ( k^N) 
\end{eqnarray*}
where $\beta_v \in \pm G$ does not depend on $j,x,y,k^M,k^N$ (but depends on $v$). 
The previous identity makes sense, and holds true, when $k^M$ is an arbitrary family of $g$ vectors in $H^M$
and $k^N$ is an arbitrary family of $h$ vectors in $H^N$. 
(Indeed, if $k^M$ is not a basis of $H^M$ or $k^N$ is not a basis of $H^N$, then both sides of this identity  are zero.) 
In particular, we obtain for any subset $P \subset \{1,\dots, 2g_+\}$ of cardinality $g-j$
\begin{eqnarray*}
&&\mathcal{R}_{N \circ M}^{\psi + \varphi } \big( \Lambda^j m_-(x) \wedge \Lambda^{f-j}  n_+(y) \big) 
\! \cdot  \! \left[\frac{\big(n_-(v_{\overline P}), n_+(y), m_+(v_P),m_-(x)\big)}{\big((-n_-,m_+)(v),\hbox{\scriptsize lift of $k$ to } H^N \oplus H^M  \big)}\right]\\
&=& \beta_v \cdot \mathcal{R}_M^{\varphi} \big( \Lambda^{g-j} m_+(v_P) \wedge \Lambda^j m_-(x) \big) 
\cdot  \mathcal{R}_N^{\psi} \big(\Lambda^{\delta\!g+j} n_-(v_{\overline P}) \wedge \Lambda^{f-j}  n_+(y)\big). 
\end{eqnarray*}
By multilinearity of the determinant and using the facts that $\dim H^M=g$ and $\dim H^N=h$, we have
\begin{eqnarray*}
1 &=&  \left[\frac{\big(-n_-(v_1)+m_+(v_1),\dots, -n_-(v_{2g_+})+m_+(v_{2g_+}),m_-(x),n_+(y)\big)}{\big((-n_-,m_+)(v),\hbox{lift of $k$ to $H^N \oplus H^M$}  \big)}\right] \\
&=& \sum_{\vert P \vert =g-j}  \varepsilon_P (-1)^{\vert \overline P \vert }
 \left[\frac{\big(m_+(v_P), n_-(v_{\overline P}), m_-(x), n_+(y)\big)}{\big((-n_-,m_+)(v),\hbox{lift of $k$ to $H^N \oplus H^M$}  \big)}\right]\\ 
&=& (-1)^{g(f+1)} \sum_{\vert P \vert =g-j} \varepsilon'_P 
\left[\frac{\big(n_-(v_{\overline P}), n_+(y), m_+(v_P),m_-(x)\big)}{\big((-n_-,m_+)(v),\hbox{lift of $k$ to $H^N \oplus H^M$}  \big)}\right].
\end{eqnarray*}
Thus we obtain identity \eqref{eq:RNM}, up to multiplication by an element of $\pm G$ not depending on $j,x,y$.
This concludes the proof of Lemma \ref{lem:functoriality_for_Reid}.

 \section{Back to the Alexander functor}   \label{sec:back_to_A}

We show in this section that the  functor $\Alex$ is an instance of the functor $\Reid$.

\subsection{A formula for the Reidemeister function} \label{subsec:Fox}

Let $M$ be a compact connected orientable $3$-manifold with connected boundary, and fix a base point $\star \in \partial M$.
Let also $\varphi:H_1(M) \to G$  be a group homomorphism with values in a multiplicative subgroup $G$ of a field $\F$.
We use the same notation as in \S \ref{subsec:R_function}, where we have introduced  $\mathcal{R}_M^\varphi$.

When it does not vanish, the Reidemeister function $\mathcal{R}_M^\varphi$  is  defined as  an alternated product of  $4$ determinants
since the $\F$-chain complex $C^\varphi(M,\star)$ has length $3$.
We now give a recipe to compute it by means of a single determinant using  Fox's free derivatives.
We consider on this purpose  a \emph{spine} $X^+$ of $M$, i.e$.$ a
$2$-dimensional subcomplex $X^+$ of a smooth triangulation of $M$ such that $M$ retracts to $X^+$ by elementary collapses; we also assume that $\star$ is a vertex of $X^+$.
(It is well known that any $3$-manifold with boundary   has a spine: see for instance \cite[Remark 1.1.5]{Matveev}.) 
Next, we choose a maximal tree  in the $1$-skeleton of $X^+$ which contains $\star$, and let $X$ be the $2$-dimensional CW-complex obtained from $X^+$ by collapsing that tree to the vertex $\star$.
Hence  $X$ has a single 0-cell~$\star$.
We denote by  $\gamma_1,\dots, \gamma_{g+r}$ the $1$-cells of $X$ and we denote by $R_1,\dots,R_r$ the $2$-cells of~$X$;
besides, each of these cells is given an arbitrary orientation.
The fundamental group $\pi_1(\Gamma)=\pi_1(\Gamma,\star)$  of  the $1$-skeleton  $\Gamma := \gamma_1 \cup \dots \cup \gamma_{g+r}$  of $X$ 
is freely generated by the oriented loops $\gamma_1, \dots , \gamma_{g+r}$, hence
the free derivatives  $\frac{\partial \ \, }{\partial \gamma_1},\dots , \frac{\partial \ \, }{\partial \gamma_{g+r}}:\Z[\pi_1(\Gamma)] \to \Z[\pi_1(\Gamma)]$ are defined.
Note that the attaching maps of the oriented $2$-cells $R_1,\dots,R_r$  define some elements $\rho_1,\dots,\rho_r \in \pi_1(\Gamma)$. 

\begin{lemma}\label{lem:Fox}
Let $\kappa_1,\dots, \kappa_g$ be oriented loops in $\Gamma$ based at $\star$
and, for all $i\in\{1,\dots,g\}$,  let $k_i \in H\simeq H_1^\varphi(X,\star)$ be the homology class of $1\otimes \widehat \kappa_i \in C_1^\varphi(X,\star)$.
Then
\begin{equation}\label{eq:Fox}
\mathcal{R}_M^\varphi(k_1\wedge \cdots \wedge k_g) 
=  \det\, \varphi\, i_* \begin{pmatrix} 
\frac{\partial \rho_1}{\partial \gamma_1}  & \cdots & \cdots& \cdots & \frac{\partial \rho_1}{\partial \gamma_{g+r}} \\
\vdots & &&& \vdots\\
 \frac{\partial \rho_r}{\partial \gamma_1} & \cdots &  \cdots & \cdots  &\frac{\partial \rho_r}{\partial \gamma_{g+r}} \\
 \frac{\partial \kappa_1}{\partial \gamma_1} & \cdots  & \cdots & \cdots  &\frac{\partial \kappa_1}{\partial \gamma_{g+r}}\\
\vdots & && &\vdots\\
 \frac{\partial \kappa_g}{\partial \gamma_1} & \cdots  & \cdots  & \cdots  & \frac{\partial \kappa_g}{\partial \gamma_{g+r}}
\end{pmatrix} .
\end{equation} 
Here the composition of $\varphi$ with the isomorphism $H_1(M)\simeq H_1(X)$ induced by the homotopy  equivalence $M \simeq X$ is still denoted by $\varphi$,
and the ring  homomorphism $i_*: \Z[\pi_1(\Gamma)] \to \Z[\pi_1(M)]$ is induced by the map $i:\Gamma \to M$ which is the inclusion $\Gamma \subset X$ composed with the homotopy equivalence $X \simeq M$.
\end{lemma}

\begin{proof}
The lemma is proved in a way similar to Milnor's result relating the Reidemeister torsion of a knot exterior to the Alexander polynomial of the knot \cite[Theorem~4]{Milnor_duality}. 
(See also \cite[Theorem II.1.2]{Turaev_book_dim_3}.)
By assumption, the pair $(M,\star)$ has the simple homotopy type of $(X^+,\star)$
and, using the  multiplicativity property of  torsions  (Theorem~\ref{th:multiplicativity}), 
it can be checked that the Reidemeister torsions of $(X,\star)$ and $(X^+,\star)$ are equal for any choice of homological bases.
Therefore we can safely replace $M$ by $X$ in our computation of $\mathcal{R}_M^\varphi$.
Thus we now consider the $\varphi$-twisted cell chain complex 
$$
C := C^{\varphi}(X,\star) = \F \otimes_{\Z[H_1(X)]} C\big(\widehat X,p_X^{-1}(\star)\big).
$$
The lifts $\widehat \gamma_1,\dots, \widehat \gamma_{g+r}$ of $\gamma_1,\dots,\gamma_{g+r}$  define a basis 
$c_1:=(1 \otimes \widehat \gamma_1,\dots, 1 \otimes \widehat \gamma_{g+r})$ of $C$ in degree $1$.
Similarly, the lifts $\widehat R_1,\dots, \widehat R_r$ of $R_1,\dots, R_r$ that contain $\widehat \star$
define a basis $c_2:=(1 \otimes \widehat R_1,\dots, 1 \otimes \widehat R_{r})$ of $C$ in degree $2$.

Let $A'$ be the square matrix with entries in $\F$ defined by the right-hand side of \eqref{eq:Fox},
and let $A$ be the $r \times (g+r)$ matrix defined by the first $r$ rows of $A'$.
Observe that  $A$ is the matrix of $\partial_2:C_2 \to C_1$ in the bases $c_2$ and $c_1$.
Since $(X,\star)$ has no relative cells in degree~$0$,
$H \simeq H_1(C)$ is the cokernel of the linear map $\F^r \to \F^{g+r}$ defined by the multiplication $v \mapsto v A$.
Assume that $\dim H >g$: then the rank of $A$ is less than~$r$, so that all the minors of $A$ of order $r$ vanish;
by expanding the determinant of $A'$ successively along its last $g$ rows, we obtain that $\det A'=0$ 
and the lemma trivially holds true in that case. Therefore we can assume that $\dim H =g$.

Observe, next, that the last $g$ rows of $A'$ give the vectors $k_1,\dots, k_g\in H\simeq H_1^\varphi(X,\star) $ as linear combinations of 
the generators $[1\otimes \widehat\gamma_1] , \dots, [1\otimes \widehat\gamma_{g+r}]$ of $H_1^\varphi(X,\star) \simeq H$.
If $k:=(k_1,\dots,k_g)$ is not a basis of $H$, then $k_1,\dots, k_g$ are linearly dependent:
since the first $r$ rows of $A'$ give a system of relations for the previous set of generators, 
we deduce that $\det A'=0$ and the lemma is trivially true  in that case too. 
Thus we can assume that $k$ is a basis of $H$.
Let $c$ be the basis of $C$ given by $c_1$ in degree $1$ and $c_2$ in degree $2$.
By Lemma~\ref{lem:homology}, the homology of $C$ is concentrated in degree $1$ and, 
for all $i\in \{1,\dots,g\}$, $1\otimes \widehat \kappa_i$ is a $1$-cycle of $C$ representing  $k_i\in H \simeq H_1(C) $.
So, by definition of the function $\mathcal{R}_M^\varphi$, we get
\begin{eqnarray}
\label{eq:computation_tau}\qquad  \  \mathcal{R}_M^\varphi(k_1 \wedge \cdots \wedge k_g)&= &\tau\left(C; c, k \right) \\
\notag &=& \det\left(\hbox{matrix of $\big(\partial_2(c_2),1\otimes  \widehat \kappa \big)$ in the basis $c_1$}\right).
\end{eqnarray}
The conclusion follows from the previous two observations.
\end{proof}

\begin{remark}
It follows from Lemma \ref{lem:Fox} that the Reidemeister function has the following \emph{integrality} property:
for all $h_1,\dots,h_g \in H_1(M,\star;\Z[H_1(M)])$, we have
$$
\mathcal{R}_M^\varphi\big(\varphi_*(h_1) \wedge \dots \wedge \varphi_*(h_g)\big) \in \varphi\big(\Z[H_1(M)]\big)
$$ 
where $\varphi_*:H_1(M,\star;\Z[H_1(M)]) \to H_1^{\varphi}(M,\star)$ is the canonical map.
\end{remark}

\subsection{Specialization of $\Reid$ to $\Alex$}

We now assume that $G$ is a finitely generated free abelian group, and we denote by $Q(G)$ the field of fractions of $\Z[G]$.
Let  $M$ be a compact connected orientable $3$-manifold with connected boundary, and fix a base point $\star \in \partial M$.
Let $\varphi:H_1(M)\to G$ be a group homomorphism:
 we denote by $\varphi_\Z:\Z[H_1(M)] \to \Z[G]$ and by $\varphi: \Z[H_1(M)]\to Q(G)$   the extensions of $\varphi$ to ring homomorphisms.  
Set 
$$
g:=g(M),  \quad H_\Z :=H_1^{\varphi_\Z}(M,\star), \quad H:=H_1^{\varphi}(M,\star).
$$

\begin{lemma} \label{lem:A_R}  
We have the following commutative diagram:
$$
\xymatrix{
\Lambda^g H_\Z \ar[d]_{\Lambda^g \iota} \ar[r]^-{\mathcal{A}_M^\varphi} & \Z[G] \ar@{^{(}->}[d] \\
\Lambda^g H \ar[r]_-{\mathcal{R}_M^\varphi} & Q(G),
}
$$
where $\iota:H_\Z \to H \simeq Q(G)  \otimes_{\Z[G]} H_\Z$ denotes the canonical map.
\end{lemma}

\begin{proof}
We proceed as in \S \ref{subsec:Fox}:  we consider a spine $X^+$ of $M$, 
and we obtain  a $2$-dimensional CW-complex $X$ with a single vertex $\star$  by collapsing a maximal tree in the $1$-skeleton of $X^+$.
The cells  of $X$ are  $\gamma_1,\dots,\gamma_{g+r}$ in dimension $1$, and $R_1,\dots,R_r$ in dimension $2$.
Orient $\gamma_1,\dots,\gamma_{g+r}$ and $R_1,\dots,R_r$ arbitrarily, and  set 
$$
C_\Z :=C^{\varphi_\Z}(X,\star), \qquad C := C^{\varphi}(X,\star) = Q(G) \otimes_{\Z[G]} C_\Z. 
$$
Since $M$ deformation retracts to $X$, $H_\Z$ is isomorphic 
to $H_1^{\varphi_\Z}(X,\star)$ so that $H_\Z$ is the cokernel of $\partial_2:C_{\Z,2} \to C_{\Z,1}$.
Let $\widehat \gamma_1,\dots,\widehat \gamma_{g+r}$ be the preferred lifts of $\gamma_1,\dots,\gamma_{g+r}$ to $\widehat X$,
and let $\widehat R_1,\dots, \widehat R_r$ be the lifts of $R_1,\dots,R_r$ that contain $\widehat \star$:
we denote by $A$ the matrix of $\partial_2$ in the bases  $\big(1 \otimes \widehat R_1,\dots, 1 \otimes \widehat R_r\big)$ 
and  $\big(1 \otimes \widehat \gamma_1,\dots,1 \otimes  \widehat \gamma_{g+r}\big)$.
This presentation matrix of the $\Z[G]$-module $H_\Z$  can be used to compute $\mathcal{A}_M^\varphi$.
Specifically, let $k_1,\dots,k_g\in H_\Z$ and assume that each $k_i$ has the form $[1 \otimes \widehat \kappa_i]$ 
where $\kappa_i$ is an oriented loop in the $1$-skeleton of $X$ based at $\star$: then
$\mathcal{A}_M^\varphi(k_1 \wedge \cdots \wedge k_g)$ is the determinant of the matrix obtained from $A$  by adding 
$g$ rows that express the vectors $1 \otimes \widehat \kappa_1, \dots, 1 \otimes \widehat \kappa_g$ 
in the basis $(1 \otimes \widehat \gamma_1,\dots,1 \otimes \widehat \gamma_{g+r})$ of $C_{\Z,1}$.
We deduce from formula (\ref{eq:computation_tau}) that
$
\mathcal{A}_M^\varphi(k_1\wedge \cdots \wedge k_g) 
= \mathcal{R}_{M}^\varphi\big(\iota(k_1)\wedge \cdots \wedge \iota(k_g)\big).
$
\end{proof}

The next theorem, which compares the Alexander functor to the Reidemeister functor, is a direct application of Lemma \ref{lem:A_R}.

\begin{theorem}\label{th:R_vs_A}
The following diagram is commutative:
$$
\xymatrix @!0 @R=1cm @C=3cm {
& \PMod_{\Z[G],\pm G} \ar@/^2pc/[dd]^-{Q(G) \otimes_{\Z[G]} (-)}  \\
\Cob_G \ar[ru]^-{\Alex}  \ar[rd]_-{\Reid}  & \\
&\PVect_{Q(G),\pm G}
}
$$
\end{theorem}

\section{Reidemeister functor and knots}  \label{sec:R_knots}

We now  compute the functor  $\Reid$ on knot exteriors and we consider, next, the situation of closed $3$-manifolds.
 In this section, we fix a field $\F$ and a multiplicative subgroup $G$ of $\F$.
 The extension of  a group homomorphism $\varphi:H \to G$ to a ring homomorphism $\Z[H]\to \F$  is  still denoted by $\varphi$.

\subsection{The abelian Reidemeister torsion of a CW-pair} \label{subsec:torsion_CW-pair}

Let $(X,Y)$ be a pair of finite CW-complexes,
and let $\varphi: \Z[H_1(X)]\to \F$ be a ring homomorphism.
We consider the $\varphi$-twisted cell chain complex $C^\varphi(X,Y)$ of the pair $(X,Y)$, which is a finite $\F$-chain complex of length $p:=\dim X$.
For every $i\in\{0,\dots,p\}$, let $n_i\geq 0$ be the number of relative $i$-cells of $(X,Y)$
and order them $\sigma_1^{(i)},\dots,\sigma_{n_i}^{(i)}$ in an arbitrary way. 
For every  cell $\sigma$ of $(X,Y)$, we also choose an orientation of $\sigma$ 
and a lift $\hat \sigma$ of $\sigma$ to the maximal abelian cover  $\widehat X$ of $X$.
Thus, we obtain a basis $c:=(c_p,\dots,c_0)$  of the $\F$-chain complex  $C^{\varphi}(X,Y)$
where, for every  $i\in\{0,\dots,p\}$, the basis of $C^{\varphi}_i(X,Y)$ is 
$
c_i:= \big(1\otimes \hat\sigma_1^{(i)},\dots,1 \otimes \hat\sigma_{n_i}^{(i)}\big).
$
Recall that the \emph{Reidemeister torsion} of the pair $(X,Y)$ with coefficients $\varphi$ is the scalar
$$
\tau^\varphi(X,Y):= \left\{\begin{array}{ll}
0 & \hbox{if} \ H^\varphi(X,Y) \neq 0,\\
\tau\big(C^\varphi(X,Y);c\big) & \hbox{if} \ H^\varphi(X,Y) = 0,
\end{array}\right.
$$
where $\tau(C;c)$ denotes the torsion of a finite  acyclic $\F$-chain complex $C$ based by $c$: see \S \ref{subsec:torsion_def}.
The reader is referred to the monograph \cite{Turaev_book} for an introduction to this combinatorial invariant.
Without further structure on the  CW-pair $(X,Y)$, the scalar $\tau^\varphi(X,Y)$
 is only defined up to multiplication by an element of $\pm \varphi(H_1(X))$.
If $Y=\varnothing$, we denote it by $\tau^\varphi(X)$.

\subsection{The Reidemeister function in genus one}

We now consider a compact connected orientable $3$-manifold $M$ with connected boundary and a group homomorphism  $\varphi:H_1(M)\to G$.
Let $\star \in \partial M$ and set $H:= H_1^{\varphi}(M,\star)$.
The next lemma relates the Reidemeister function $\mathcal{R}_M^\varphi$  to the Reidemeister torsion $\tau^{\varphi}(M)$ in genus one.

\begin{lemma}\label{lem:def_one_for_Reid}
Assume that $g(M)=1$ and that $\varphi$ is not trivial. Then, for any $k\in H$, 
\begin{equation}\label{eq:R_tau}
 \mathcal{R}_M^\varphi(k) = \tau^{{\varphi}}(M) \cdot \partial_*(k).
\end{equation}
Here $\partial_*: H\to \F$ is the connecting homomorphism $ H_1^{\varphi}(M,\star) \to H_0^{\varphi}(\star)$
in the long exact sequence of the pair $(M,\star)$, 
followed by the canonical isomorphism $H_0^\varphi(\star)\simeq \F$.
\end{lemma}

\begin{proof}
Consider a cell decomposition of $M$ where $\star$ is a $0$-cell.
The  short exact sequence of $\F$-chain complexes 
\begin{equation}\label{eq:ses}
0 \longrightarrow \underbrace{C^{\varphi}(\star)}_{C':=} \longrightarrow 
\underbrace{C^{\varphi}(M)}_{C:=} \longrightarrow 
\underbrace{C^{\varphi}(M,\star)}_{C'':=} \longrightarrow 0
\end{equation}
induces the following long exact sequence in homology:
\begin{equation}\label{eq:les}
 0 \longrightarrow 0 \longrightarrow 0 \longrightarrow 0 \longrightarrow 
H_2^{\varphi}(M)  \longrightarrow  H_2^{\varphi}(M,\star)  \to
\end{equation}
$$
\to 0 \longrightarrow H_1^{\varphi}(M)  \longrightarrow  H_1^{\varphi}(M,\star)  \stackrel{\partial_*}{\longrightarrow}  
H_0^{\varphi}(\star) \longrightarrow   H_0^{\varphi}(M) \longrightarrow 0
$$
We regard (\ref{eq:les}) as an acyclic $\F$-chain complex $\mathcal{H}$ of length $12$:
let $(h',h,h'')$ be the basis of $\mathcal{H}$ obtained by choosing bases $h',h,h''$ of $H(C'),H(C),H(C'')$ in each degree.
We choose an orientation and a lift to $\widehat M$ for every cell  of $M$ and, for all  $i\in\{0,\dots,3\}$,
 we order the $i$-cells in an arbitrary way. Thus, we obtain bases $c',c,c''$ of the complexes $C',C,C''$, respectively, which are compatible in the sense of \S \ref{subsec:torsion_def}.
By the multiplicativity property of  torsions (see Theorem \ref{th:multiplicativity}),    we obtain
\begin{equation}\label{eq:multiplicativity}
\tau(C;c,h) = \varepsilon \cdot  \tau(C';c',h') \cdot \tau(C'';c'',h'') \cdot \tau\big(\mathcal{H};(h',h,h'')\big)
\end{equation}
where $\varepsilon$ is a sign  independent of $h,h',h''$.
If $H_2^{\varphi}(M)\neq 0$, then $\tau^{\varphi}(M)=0$ by definition,
but (\ref{eq:les})   gives   $H_2^{\varphi}(M,\star)\neq 0$
and Lemma \ref{lem:homology} implies that $\dim H_1^{\varphi}(M,\star)>g(M)$:
hence $\mathcal{R}_M^\varphi=0$ by definition and (\ref{eq:R_tau}) trivially holds true.
Therefore we can assume that $H_2^{\varphi}(M)= 0$.

Besides  $H_0^{\varphi}(M)= 0$ since $\varphi$ is non-trivial:
the fact that $\chi(M)=1-g(M)$ is zero implies that $H_1^{\varphi}(M)= 0$ as well.
Thus the chain complex $\mathcal{H}$ defined by (\ref{eq:les}) is concentrated in degrees $2$ and $3$. 
Let $k\in H \setminus \{0\}$ which defines a basis $h''$ of $H(C'')$,
and let $h'$ be the basis of $H(C')$ defined by the canonical generator of $H_0^{\varphi}(\star)$.
Then we obtain
$$
\tau\big(\mathcal{H};(h',h,h'')\big) = \left[\partial_*(h''_1)/ h'_0\right]^{(-1)^{2+1}} = \partial_*(k)^{-1}.
$$
Besides we have $\tau(C';c',h')=1$ by our choices of $c'$ and $h'$. 
We conclude thanks to (\ref{eq:multiplicativity}) that
$
\tau^{\varphi} (M) = \varepsilon \cdot  \mathcal{R}_M^\varphi(k) \cdot \partial_*(k)^{-1}.
$
\end{proof}

\begin{remark} \label{rem:genus_0}
If $g(M)=0$ and $\varphi$ is not trivial, then $\mathcal{R}^\varphi_M: \F = \Lambda^0 H \to \F$ is the zero map.
Indeed, pick an oriented loop $\alpha$ in $M$ based at $\star$ such that $\varphi([\alpha])\neq 1$;
then $\partial_*: H \to \F$ does not vanish on $[\hat \alpha]$ and it follows that $\dim H > g(M)$.
\end{remark}

\subsection{The functor $\Reid$ on knot exteriors}

Let $K$ be an oriented knot in a closed connected oriented  $3$-manifold $N$,
and denote by $M_K$ the complement of an open tubular neighborhood of $K$ in $N$.
We assume given a group homomorphism $\varphi_K:M_K \to G$ and an oriented closed curve $\lambda \subset \partial M_K$ such that $\varphi_K([\lambda])\neq 1$. 
Thus the Reidemeister torsion $\tau^{\varphi_K}(M_K) \in \F/\pm G$ is defined.

We make $M_K$  a morphism $1\to 0$ in the category $\Cob$ by choosing a boundary-parametrization  $ m: F(1,0) \to \partial M_K$,
such that $\lambda_- := m^{-1}(\lambda)$  is contained in the bottom surface $F_1$ and goes through the base point $\star$. Set  $H_- := H_1^{\varphi_K m_- }(F_1,\star)$.
The following proposition, which is easily deduced from Lemma \ref{lem:def_one_for_Reid}, shows that 
the topological invariants $\tau^{\varphi_K}(M_K)$ and $\Reid(M_K,\varphi_K)$ are equivalent.

\begin{proposition} \label{prop:Reidemeister_knot}
With the above notation and for any $h\in \Lambda^i H_-$, we have
$$
\Reid(M_K,\varphi_K)(h) = \left\{\begin{array}{ll} \tau^{\varphi_K}(M_K) \cdot \partial_*(h) & \hbox{if } i=1,\\
0 & \hbox{otherwise},
\end{array}\right.
$$
where $\partial_*:H_-\to \F$ is the connecting homomorphism for the pair $(F_1,\star)$.
In particular,  we have $\tau^{\varphi_K}(M_K) = \Reid(M_K,\varphi_K)\big(\big[\, \widehat \lambda_-\big]\big)/(\varphi_K([\lambda])-1)$.
\end{proposition}

\begin{example}
If $G$ is the infinite cyclic group generated by $t$, $N$ is a homology $3$-sphere and  $\varphi_K$ maps the oriented meridian $\mu$ of $K$ to $t$,
then we know from \cite{Milnor_duality} that $\tau^{\varphi_K}(M_K)=\Delta(K)/(t-1)$. 
Thus we recover Proposition \ref{prop:Alexander_knot} by taking $\lambda:=\mu$.
\end{example}

\subsection{The situation of closed $3$-manifolds } \label{subsec:Rf_vs_Rt}

Let $N$ be a closed connected orientable $3$-manifold,  and let $\varphi:H_1(N) \to G$ be a non-trivial group homomorphism.
We  wish to compute the Reidemeister torsion $\tau^{\varphi}(N)$ with coefficients $\varphi:\Z[H_1(N)] \to \F$ 
 from the Reidemeister functor $\Reid$. 
 For this, we have to transform $N$ into a cobordism. 
Note that removing an open $3$-ball $B$ from $N$ and regarding $N \setminus B$ as an element of $\Cob(0,0)$ is not fruitful,
since the functor $\Reid$ maps this morphism to zero (see Remark \ref{rem:genus_0}).

We proceed in the following (rather indirect) way. 
Choose a knot $K \subset N$ such that $\varphi([K])\neq 1$.
Consider the complement $M_K$ of an open tubular neighborhood of $K$ in $N$,
and fix a parallel $\rho \subset \partial M_K$ of $K$.
Let  $\varphi_K:H_1(M_K) \to G$ be the homomorphism obtained  from $\varphi$ by restriction to $M_K \subset N$.
Make $M_K$ a morphism $1 \to 0$ in  $\Cob$ by choosing a boundary-parametrization $m: F(1,0) \to \partial M_K$
such that $\rho_- := m^{-1}({\rho})$ is contained in the bottom surface $F_1$ and $\star \in \rho_-$.

\begin{proposition}\label{prop:Rf_vs_Rt}
With the above notation, we have
$$
\tau^{\varphi}(N)= \frac{\Reid(M_K,\varphi_K)([\widehat \rho_-])}{(\varphi([K])-1)^2} \ \in \F/\pm G.
$$
\end{proposition}

\begin{proof}
There is a formula describing (under certain circumstances) how the abelian Reidemeister torsion changes when  a solid torus is glued along a $3$-manifold with toroidal boundary: see \cite[\S VII.1]{Turaev_book_dim_3}.
This formula applies to our situation and gives
$$
\tau^{\varphi_K}(M_K)= (\varphi([K])-1) \cdot \tau^{\varphi}(N).
$$
We conclude by applying Proposition \ref{prop:Reidemeister_knot} to $\lambda:= \rho$.
\end{proof}

As an application, we relate the functor $\Alex$ to the Alexander polynomial of closed $3$-manifolds.
Thus, we now assume that $G$ is a finitely generated free abelian group and we take $\F:=Q(G)$.
We consider the \emph{Alexander polynomial} of $N$ with coefficients $\varphi$, namely
$$
\Delta^\varphi(N) = \Delta_0\, H_1^{\varphi_\Z}(N) \ \in \Z[G]/\!\pm G
$$
where $\varphi_\Z: \Z[H_1(N)] \to \Z[G]$ is the extension of $\varphi:H_1(N) \to G$.

\begin{proposition}\label{prop:Alexander_closed}
With the above notation, we have
$$
\Delta^\varphi(N) =
\left\{\begin{array}{ll}
{\displaystyle \frac{\Alex(M_K,\varphi_K)([\widehat \rho_-])}{(\varphi([K])-1)^2}} & \hbox{if} \ \rank \varphi(H_1(N)) \geq 2, \\
{\displaystyle \frac{\Alex(M_K,\varphi_K)([\widehat \rho_-])}{(t^{n-1}+\cdots + t +1)^2}} & \hbox{if} \ \rank \varphi(H_1(N)) = 1.
\end{array}\right.
$$
In the second case, $t\in \varphi(H_1(N))$ is a generator and $n\in \N$ is such that $\varphi([K])=t^n$.
\end{proposition}

\begin{proof} Proposition \ref{prop:Rf_vs_Rt} and Theorem \ref{th:R_vs_A} give 
\begin{equation}\label{eq:presque}
\tau^{\varphi}(N)= \frac{\Reid(M_K,\varphi_K)([\widehat \rho_-])}{(\varphi([K])-1)^2}
=  \frac{\Alex(M_K,\varphi_K)([\widehat \rho_-])}{(\varphi([K])-1)^2} \ \in Q(G)/\pm G.
\end{equation}
Besides, according to  \cite{Turaev_Alexander}, we have
\begin{equation}\label{eq:Turaev_Alexander}
\tau^{\varphi}(N) = 
\left\{\begin{array}{ll}
{\displaystyle \Delta^\varphi(N)} & \hbox{if} \ \rank \varphi(H_1(N)) \geq 2, \\
{\displaystyle {\Delta^\varphi(N)}/{(t-1)^2}} & \hbox{if} \ \rank \varphi(H_1(N)) = 1.
\end{array}\right.
\end{equation}
We conclude by combining \eqref{eq:presque} to \eqref{eq:Turaev_Alexander}.
\end{proof}

\section{The monoid of homology cobordisms} \label{sec:homology_cobordisms}

In this section, we fix an integer $k\geq 1$,  an abelian group $G$ and a group homomorphism $\psi: H_1(F_k) \to G$.
We shall compute the  functors $\Alex$ and $\Reid$ on the monoid of homology cobordisms over the surface $F_k$.

\subsection{Homology cobordisms} \label{subsec:homology_cobordism}

A \emph{homology cobordism} over $F_k$ is a morphism $M:k\to k$ in the category $\Cob$ such that
$m_\pm:H_1(F_k) \to H_1(M)$ is an isomorphism. 
The set of   equivalence classes   of homology cobordisms defines a submonoid
$$
\C(F_k) \subset \Cob(k,k).
$$
We restrict ourselves to homology cobordisms $M$ such that the composition
$$
\xymatrix{
H_1(F_k) \ar[r]^{m_{-}}_-\simeq & H_1(M) \ar[r]^{m_{+}^{-1}}_-\simeq & H_1(F_k) \ar[r]^-\psi &  G
}
$$
coincides with $\psi$. Thus we obtain a submonoid 
$$
\C^\psi(F_k) \subset \C(F_k),
$$
which we also view as a submonoid of $\Cob_G\big((k,\psi),(k,\psi)\big)$ 
by equipping every cobordism $M$ of the above form  with the homomorphism $\psi:=\psi  \circ m_{-}^{-1}=\psi  \circ m_{+}^{-1}:H_1(M) \to G$.

\begin{example}
A \emph{homology cylinder} is a homology cobordism $M$ over $F_k$ such that $m_-=m_+:H_1(F_k) \to H_1(M)$.
Homology cylinders constitute a submonoid $\mathcal{IC}(F_k)$ of $\C(F_k)$ such that
$\mathcal{IC}(F_k)  \subset \C^\psi(F_k)$, whatever $\psi$ is.
\end{example}  

\subsection{The Magnus representation} \label{subsec:Magnus}

Assume now that  $G$ is a multiplicative subgroup of a field $\F$.
The extension of $\psi:H_1(F_k) \to G$ to a ring homomorphism $\Z[H_1(F_k)] \to \F$ is still denoted by $\psi$.
We set $$H^\psi  := H_1^{\psi}(F_k,\star)$$
and,  when we are given an $M \in \mathcal{C}^\psi(F_k)$,  we  denote $H := H_1^{\psi  }(M,I)$.
The fact that the map $m_{\pm}: H_1(F_k) \to H_1(M)$ is an isomorphism of abelian groups implies that   $m_{\pm} :  H^\psi \to  H$
is an isomorphism of $\F$-vector spaces. (See \cite[Proposition 2.1]{KLW} for  a similar statement.) Consequently, we are allowed to set
$
r^{\psi}(M) := m_+^{-1} \circ m_- : H^\psi \to H^\psi.
$
This results in a monoid homomorphism 
$$
r^{\psi}: \C^\psi(F_k) \longrightarrow \Aut\!\big( H^\psi \big),
$$
which is called the  \emph{Magnus representation}.  See  \cite{Sakasai_survey} for a survey of this invariant.

\subsection{The restriction of $\Reid$ to homology cobordisms} \label{subsec:R_C}

The Reidemeister functor restricts to a monoid homomorphism
$$
\Reid: \C^\psi(F_k) \longrightarrow \PVect_{\F,\pm G}\big(\Lambda H^\psi,\Lambda H^\psi\big).
$$
We now compute this projective representation of the monoid $\C^\psi(F_k)$. 

\begin{proposition} \label{prop:homology_cobordisms_Reidemeister}
For any $M \in \C^\psi(F_k)$ with top surface $\partial_+ M$,  we have 
$$
\Reid(M,\psi) = \tau^{\psi}(M,\partial_+M) \cdot \Lambda\big( r^{\psi} (M)\big) : \Lambda H^\psi \longrightarrow \Lambda H^\psi
$$
where  $\tau^{\psi} (M, \partial_+M) $ is  the   Reidemeister torsion of $(M, \partial_+M)$ as defined in \S \ref{subsec:torsion_CW-pair}.
\end{proposition}

\begin{proof}
We shall prove a slightly more general statement: let $\psi_\pm : H_1(F_k) \to G$ be any group homomorphisms
and assume that $M \in \C(F_k)$ is a cobordism such that $\psi_- \circ m_-^{-1}= \psi_+ \circ m_+^{-1}: H_1(M) \to G$. Then we claim that
\begin{equation}\label{eq:more_general}
\Reid(M,\psi) = \tau^{\psi}(M,\partial_+M) \cdot \Lambda (m_+^{-1} \circ m_-) : \Lambda H_- \longrightarrow \Lambda H_+
\end{equation}
where $H_\pm := H_1^{\psi_\pm} (F_k,\star)$ and $\psi:= \psi_\pm \circ m_\pm^{-1}$.
(The proposition  is the particular case where $\psi_+ = \psi_-: H_1(F_k) \to G$.)

To prove this claim, we set $g:=g(M)=2k$, $H := H_1^{\psi }(M,I)$ and let $h=(h_1,\dots,h_g)$ be a basis of $H$. 
In order to compute  $\mathcal{R}_M^\psi(h_1\wedge \cdots \wedge h_g)$, we consider the short exact sequence of $\F$-chain complexes:
\begin{equation} \label{eq:ses_relative}
0 \longrightarrow {\underbrace{C^{\psi_+}(F_k,\star)}_{C':=} }  \stackrel{m_+}{\longrightarrow} {\underbrace{C^{\psi}(M,\star)}_{C:=}} 
\longrightarrow {\underbrace{C^{\psi}(M,\partial_+M)}_{C'':=} }\longrightarrow  0
\end{equation}
The complex $C''$ is acyclic while  $C'$ and $C$ have their homology concentrated in degree~$1$.
Therefore, the long exact sequence  in homology $\mathcal{H}$ induced by \eqref{eq:ses_relative} is concentrated in degrees $4$ and $5$ 
where it reduces to the map  $m_+: H_+ =H_1(C') \to H_1(C)=H$.

There exists a wedge of  circles $S_1 \vee \cdots \vee S_g$ based at $\star$ onto which  the surface $F_k$ retracts by elementary collapses.
Let $h'=(h'_1,\dots,h'_{g})$ be the basis of $H_+$ obtained by lifting each of the loops $S_1,\dots,S_g$ to the maximal abelian cover of $F_k$. 
Then we have
$$
\tau(C';c',h')  \in \pm G \subset  \F
$$
for any choice of ordered/oriented lifts of the relative cells of $(F_k,\star)$ inducing a basis $c'$ of $C'$. 
Besides, by the multiplicativity property of torsions (see Theorem \ref{th:multiplicativity}), we have
$$
\tau(C;c,h) = \varepsilon \cdot \tau(C';c',h') \cdot \tau(C'';c'') \cdot \tau\big(\mathcal{H};(h',h)\big) \ \in \F \setminus \{0\}
$$
for some appropriate choices  of ordered/oriented lifts of the relative cells,  which result in bases $c',c,c''$ of the chain complexes.
Here $\varepsilon$ is a sign not depending on $h$ and  $\mathcal{H}$ is regarded as an acyclic $\F$-chain complex based by $(h',h)$. We deduce that
\begin{eqnarray*}
\mathcal{R}_M^{\psi}(h_1\wedge \cdots \wedge h_g) \ = \  \tau(C;c,h) 
&  =&  \tau^{\psi}(M,\partial_+ M) \cdot  [m_+(h')/h]^{(-1)^{4+1}} \\
& = & \tau^{\psi}(M,\partial_+ M) \cdot  [h/m_+(h')].
\end{eqnarray*}
(Here the  identities are up to multiplication by an element of $\pm G$ not depending on $h$.)

To proceed, we consider any integer $j\geq 0$ and any $x\in \Lambda^j H_-$. 
Let $\vol: \Lambda^{g} H_+ \to \F$ be the volume form defined by $\vol(h'_1 \wedge \cdots \wedge h'_g)=1$. (Note that $\vol$ is integral.)
Then, for any $y\in \Lambda^{g-j} H_+$, we have
\begin{eqnarray*}
\vol \big(\Reid(M,\psi)(x)\wedge y\big) 
&=& \mathcal{R}_M^{\psi}\big( \Lambda^j m_-(x) \wedge \Lambda^{g-j} m_+(y)\big) \\
& = & \tau^{\psi}(M,\partial_+ M) \cdot  \left[  \left. \big(\Lambda^j m_-(x) \wedge \Lambda^{g-j}m_+(y)\big) \right/m_+(h')\right] \\
& = & \tau^{\psi}(M,\partial_+ M) \cdot  \left[  \left. \big(\Lambda^j (m_+^{-1} m_-)(x) \wedge y \big) \right/ h' \right] \\
& = & \tau^{\psi}(M,\partial_+ M) \cdot  \vol\left( \Lambda^j (m_+^{-1} m_-) (x) \wedge y \right).
\end{eqnarray*}
We conclude that  $\Reid(M,\psi)(x) = \tau^{\psi}(M,\partial_+ M) \cdot \Lambda^j (m_+^{-1} m_-) (x)$
 up to multiplication by an element of $\pm G$ not depending on $x$, which proves \eqref{eq:more_general}.
\end{proof}

\subsection{The restriction of $\Alex$ to  homology cobordisms}

Assume now that the abelian group $G$ is  finitely generated and free, and  assume that $\F:=Q(G)$.
We denote by $\psi_\Z:\Z[H_1(F_k)] \to \Z[G]$ the extension of $\psi:H_1(F_k) \to G$ to a ring homomorphism
and we set $H^\psi_\Z := H_1^{\psi_\Z}(F_k,\star)$.
The Alexander functor restricts  to a monoid homomorphism
$$
\Alex: \C^\psi(F_k) \longrightarrow \PMod_{\Z[G],\pm G}\big(\Lambda H^\psi_\Z,\Lambda H^\psi_\Z\big).
$$
This projective representation of the monoid $\C^\psi(F_k)$ is computed as follows.

\begin{proposition} \label{prop:homology_cobordisms_Alex}
For any $M \in \C^\psi(F_k)$, 
we have the commutative diagram
$$
\xymatrix@!C{
\Lambda H^\psi_\Z \ar[rrr]^-{\Alex(M,\psi) } \ar@{^{(}->}[d] && & \Lambda H^\psi_\Z \ar@{^{(}->}[d]  \\
\Lambda H^\psi  \ar[rrr]_{\Delta^\psi(M,\partial_+M) \cdot \Lambda  r^{\psi} (M)} &&& \Lambda H^\psi
}
$$
where  $\Delta\!^\psi(M, \partial_+M) $ is  the Alexander polynomial of the pair $(M,\! \partial_+M)$ as defined in \S \ref{subsec:topological_pair}.
\end{proposition}
 
\begin{proof}
The proposition can be proved directly from the definition of  $\Alex$, using an appropriate presentation of the $\Z[G]$-module $H_1^{\psi_\Z} (M,I)$.
It also follows from Theorem \ref{th:R_vs_A}, Proposition \ref{prop:homology_cobordisms_Reidemeister} and the fact that
$$
\tau^{\psi} (M, \partial_+M) = \Delta^{\psi} (M, \partial_+M)  \in \Z[G]/\pm G.
$$
The latter identity is shown using the fact that $M$ collapses, relatively to $\partial_+ M$,  onto a cell complex having only $1$-cells and $2$-cells in an equal number.
(For instance, consider the CW-complex resulting from a handle decomposition of $M$ as discussed in \S \ref{subsec:Heegaard}.)
Thus, the computation of both invariants $\tau^{\psi} (M, \partial_+M)$ and  $\Delta^{\psi} (M, \partial_+M)$ reduces to the determinant of a  same matrix.
(See \cite[Lemma 3.6]{FJR} for instance.)
\end{proof}

\begin{example}
Assume that $G:=\{1\}$ is the trivial group. Then $\C^\psi(F_k)= \C(F_k)$. Moreover $\Z[G]=\Z$ and $Q(G)=\Q$, so that  $H^\psi_\Z=H_1(F_k)$ and $H^\psi=H_1(F_k;\Q)$.
Note that $\Delta^\psi(M,\partial_+M)=1$ since $H_1^{\psi_\Z}(M,\partial_+M)=H_1(M,\partial_+M)$ is trivial in that case.
It~follows from Proposition \ref{prop:homology_cobordisms_Alex}  that $\Alex(M,\psi):\Lambda H_1(F_k) \to \Lambda H_1(F_k)$ 
is induced by the isomorphism $ (m_+)^{-1} m_-:H_1(F_k)\to H_1(F_k)$. 
\end{example}

\begin{remark}
If two cobordisms $M,M'\in \C^\psi(F_k)$ are homology cobordant, 
then we have $r^{\psi}(M)=r^{\psi}(M')$ (see  \cite[Theorem 3.6]{Sakasai_Magnus1}), but it may happen that $\Delta^\psi(M,\partial_+M) \neq  \Delta^\psi(M',\partial_+M')$
(see \cite[Lemma 3.15]{MM} for an example). It follows from Proposition~\ref{prop:homology_cobordisms_Alex}  
that the restriction of $\Alex$ to $\C^\psi(F_k)$ is \emph{stronger} than the  representation~$r^\psi$.
\end{remark}

\section{Computations with Heegaard splittings} \label{sec:Heegaard}

Let $G$ be a multiplicative subgroup of a field $\F$.
We give a simple recipe to compute the functor $\Reid= \Reid_{\F, G}$  using Heegaard splittings of cobordisms.
In this section, the extension of  a group homomorphism $\rho:H \to G$ to a ring homomorphism $\Z[H]\to \F$  is still denoted by $\rho$.

\subsection{Heegaard splittings}  \label{subsec:Heegaard}

In order to obtain concrete formulas for the functor $\Reid$, 
it is convenient to fix compatible systems of ``meridians and parallels'' on the model surfaces.
Specifically, we  choose on the model surface $F_1$ a \emph{meridian} $\alpha$ and a \emph{parallel}~$\beta$, which means the following:
 $\alpha$ and $\beta$ are oriented  simple closed curves in the interior of $F_1$  meeting transversely at a single point with homological intersection $[\alpha] \bullet [\beta] =+1$.
Then the identification between $F_1 \sharp_\partial \cdots  \sharp_\partial F_1$ and $F_k$
induces, for any integer $k\geq 1$, a \emph{system of meridians and parallels} $(\alpha_1,\dots,\alpha_k,\beta_1,\dots,\beta_k)$ on the surface $F_k$.

For any  $k\geq 0$, we denote by $C_0^k \in \Cob(0,k)$ the cobordism obtained from $F_k \times [-1,1]$ 
by attaching $k$ $2$-handles along  the curves $\alpha_1\times \{-1\},\dots,\alpha_k\times \{-1\}$. 
Similarly, let $C_k^0 \in \Cob(k,0)$ be the cobordism obtained from $F_k \times [-1,1]$ 
by attaching $k$ $2$-handles along the curves  $\beta_1\times \{1\},\dots,\beta_k\times \{1\}$. Observe that $ C_k^0 \circ C^k_0 =C_0^0 \in \Cob(0,0)$ is the $3$-dimensional ball $F_0 \times [-1,1]$.
Thus we shall refer to $C_k^0$ and $C^k_0$ as the \emph{upper} and \emph{lower} handlebodies, respectively.
(Clearly, these notions depend on the above choice of meridians and parallels.) 

Let also $\mcg(F_k)$ be the \emph{mapping class group} of the surface $F_k$, 
which consists of isotopy classes of (orientation-preserving)  homeomorphisms $f:F_k \to F_k$ fixing $\partial F_k$ pointwisely.
The \emph{mapping cylinder} construction, which associates to any such homeomorphism $f$ the cobordism
$$
\mcyl(f):= \big(F_k \times [-1,1], (f \times \{-1\}) \cup  (\partial F_k \times \Id) \cup (\Id\times \{1\}) \big),
$$
defines an embedding $\mcyl: \mcg(F_k) \to \C(F_k)$
of the mapping class group into the monoid of homology cobordisms (see \S \ref{subsec:homology_cobordism}).

Let $M \in \Cob(g_-,g_+)$ be an arbitrary cobordism. 
By elementary Morse theory, the $3$-manifold underlying $M$
can be obtained from the trivial cobordism $F_{g_+} \times [-1,1]$ by attaching simultaneously some $1$-handles (say, $r_+\geq 0$) along the ``bottom surface'' $F_{g_+} \times \{-1\}$,
 and by attaching subsequently some $2$-handles (say, $r_-\geq 0$) along the new ``bottom surface.''
We obtain in that way a \emph{Heegaard splitting} of $M$, i.e$.$ a decomposition in the monoidal category $\Cob$ of the form
\begin{equation}\label{eq:Heegaard}
M = \left( C^0_{r_+} \otimes \Id_{g_+}  \right) \circ \mcyl(f) \circ \left( C_0^{r_-} \otimes \Id_{g_-} \right)
\end{equation}
where  $g_++r_+=g_-+r_-$  and  $f\in \mcg(F_{g_\pm + r_\pm})$. 
See \cite[Theorem 5]{Kerler}.

\subsection{Computation of $\Reid$ with Heegaard splittings} \label{subsec:computation_Heegaard}

We now assume that the above cobordism $M$  comes with a group homomorphism $\varphi:H_1(M) \to G$:
$$
(M,\varphi) \in \Cob_G\big((g_-,\varphi_-),(g_+,\varphi_+)\big).
$$
The Heegaard splitting \eqref{eq:Heegaard} of $M$ induces a decomposition in the monoidal category $\Cob_G$ by endowing  each submanifold $S$ of that decomposition
with the group homomorphism $\bar \varphi:H_1(S)\to G$ obtained by restricting $\varphi$ to $S \subset M$. 
Hence we obtain
$$
\Reid(M,\varphi) = \Big(\Reid\big(C^0_{r_+}, \bar \varphi\big) \otimes \Id_{\Lambda H_+} \Big) \circ 
\Reid\big(\mcyl(f),\bar\varphi\big) \circ \Big( \Reid\big(C_0^{r_-}, \bar\varphi\big) \otimes \Id_{\Lambda H_-} \Big)
$$
where $H_\pm := H_1^{\varphi_\pm}(F_{g_\pm},\star)$ and the symbol $\bar \varphi$ denotes a representation in $G$  induced by~$\varphi$.
Thus the computation of $\Reid(M,\varphi)$ reduces to three cases: upper handlebodies, lower handlebodies and mapping cylinders.

To describe the values of $\Reid$ in those three cases, we need to fix further notation.
Let $k\geq 0$ be an integer and let $\psi:H_1(F_k) \to G$ be a group homomorphism.
We assume that, in our model surface $F_1$,  the intersection point $\alpha \cap \beta$ is connected by an arc to the base point $\star \in \partial F_1$:
hence the curves  $\alpha_1,\dots,\alpha_k,\beta_1, \dots, \beta_k$ are now viewed as oriented loops based at $\star\in \partial F_k$.
We denote by $(a_1^\psi,\dots,a_k^\psi,b_1^\psi,\dots,b_k^\psi)$ the basis of $H_1^\psi(F_k,\star)$ 
obtained by lifting these loops to the maximal abelian cover: 
\begin{equation} \label{eq:a_b}
\forall i=1,\dots,k, \quad a_i^\psi := \big[1 \otimes \widehat \alpha_i  \big], \ b_i^\psi:= \big[1 \otimes \widehat \beta_i  \big].
\end{equation}
Then the space $ \Lambda H_1^{\psi}(F_k,\star)$ can be identified to $\Lambda A_k^\psi \otimes \Lambda B_k^\psi$ where 
$A_k^\psi := \langle a_1^\psi,\dots, a_k^\psi\rangle$ and $B_k^\psi := \langle b_1^\psi,\dots, b_k^\psi \rangle$
are the subspaces of $H_1^{\psi}(F_k,\star)$ corresponding to meridians and parallels, respectively.

\begin{lemma}\label{lem:upper_handlebody}
Let  $\psi:H_1(C^0_k)\to G$ be a group homomorphism
and let $\psi_-:H_1(F_k)\to G$ be the restriction of $\psi$ to $F_k\subset \partial C^0_k$. Then the linear map
$$
\Reid(C_k^0,\psi) : \Lambda H_1^{\psi_-}(F_k,\star) \longrightarrow \F
$$
is trivial on $\Lambda^i A_k^{\psi_-} \otimes \Lambda^j B_k^{\psi_-}$ if $i\neq k$ or $j\neq 0$, and it sends $a_1^{\psi_-}  \wedge \cdots \wedge a_k^{\psi_-} $ to $1$.
\end{lemma}

\begin{proof}
Set $N:= C_k^0 \in \Cob(k,0)$.
Since $\Reid(N,\psi)$ has degree $-k$, it must be trivial on $\Lambda^r H_1^{\psi_-}(F_k,\star)$ for $r\neq k$.
It remains to compute
\begin{equation}\label{eq:R_R}
\Reid(N,\psi)(x_1 \wedge \cdots \wedge x_k)
= \mathcal{R}_N^\psi \big(n_-(x_1) \wedge \cdots \wedge n_-(x_k)\big)
\end{equation}
for any $x_1,\dots,x_k \in  H_1^{\psi_-}(F_k,\star)$.
If one of the $x_i$'s  belongs to $B_k^{\psi_-}$, the right-hand side of \eqref{eq:R_R} is zero since, for all $j\in \{1,\dots,k\}$, $\beta_j$ bounds a disk in $N$ so that  $n_-(b_j^{\psi_-} )=0$.
So, we can assume that $x_1 \wedge \cdots \wedge x_k = a_1^{\psi_-}  \wedge \cdots \wedge a_k^{\psi_-}$.
In this case, we apply Lemma~\ref{lem:Fox} to the obvious spine $X=X^+$ of $N$:
the spine $X$ is a wedge of circles whose  $1$-cells $\gamma_1,\dots,\gamma_k$   are obtained by ``pushing'' the curves $\alpha_1,\dots,\alpha_k$ in the interior of~$N$.
We deduce that the right-hand side of \eqref{eq:R_R} is equal to $1$.
\end{proof}

\begin{lemma}\label{lem:lower_handlebody}
Let  $\psi:H_1(C_0^k)\to G$ be a group homomorphism
and let $\psi_+:H_1(F_k)\to G$ be the restriction of $\psi$ to $F_k\subset  \partial C_0^k$. Then the linear map
$$
\Reid(C^k_0,\psi) : \F \longrightarrow \Lambda H_1^{\psi_+}(F_k,\star) 
$$
sends the scalar $1$ to the multivector $a_1^{\psi_+} \wedge \cdots \wedge a_k^{\psi_+}$.
\end{lemma}

\begin{proof}
Set $(v_1,\dots,v_k, v_{k+1}, \dots, v_{2k}):= (a_1^{\psi_+},\dots,a_k^{\psi_+},b_1^{\psi_+},\dots,b_k^{\psi_+})$ and  let $\omega$ be the volume form on $H_1^{\psi_+}(F_k,\star)$ defined by $\omega(v_1\wedge \cdots \wedge v_{2k})=1$.
We denote $N:= C^k_0 \in \Cob(0,k)$ and write
$$
\Reid(N,\psi)(1)  = \sum_P z_P \cdot v_P \in \Lambda^k  H_1^{\psi_+}(F_k,\star)
$$
where $P$ runs over $k$-element subsets of $\{1,\dots,2k\}$ and $v_P$ is the wedge of the $v_p$'s for all $p\in P$.
For any $k$-element subset  $P \subset \{1,\dots,2k\}$, we have
\begin{equation}\label{eq:R_R...again}
\varepsilon_P \cdot z_{P} = \vol \left(\Reid(N,\psi)(1) \wedge v_{\overline P} \right)
=  \mathcal{R}_N^\psi \big( \Lambda^k n_+(v_{\overline P})\big)
\end{equation}
where $\overline P$ is the complement of $P$ and $\varepsilon_P$ is the signature of the permutation $P \overline P$.
To compute the right-hand side of \eqref{eq:R_R...again}, we apply Lemma \ref{lem:Fox} to the obvious  spine $X=X^+$ of $N$:
the spine $X$ is a wedge of circles whose $1$-cells $\gamma_1,\dots,\gamma_k$ are obtained by ``pushing'' the curves $\beta_1,\dots,\beta_k$ in the interior of $N$.
We obtain that $\mathcal{R}_N^\psi \big( \Lambda^k n_+(v_{\overline P})\big)$ is trivial except if $\overline P=\{k+1,\dots,2k\}$, 
in which case it takes the value~$1$.
We conclude that $z_P= 1$ if $P = \{1,\dots,k\}$ and $z_P=0$ otherwise.
\end{proof}

\begin{lemma}\label{lem:mapping_cylinder}
Let $f\in \mcg(F_k)$ and let $\psi_\pm: H_1(F_k) \to G$ be group homomorphisms such that $\psi_- =\psi_+ \circ f$.  
Denote by $\psi: H_1(F_k \times [-1,1]) \to G$ the isomorphism $\psi_+ \circ \operatorname{pr}$, 
where $\operatorname{pr}: F_k \times [-1,1] \to F_k$ is the cartesian projection. Then
$$
\Reid\big(\mcyl(f),\psi \big): \Lambda  H_1^{\psi_-}(F_k,\star) \longrightarrow \Lambda   H_1^{\psi_+}(F_k,\star) 
$$
is induced by the isomorphism $f: H_1^{\psi_-}(F_k,\star) \to H_1^{\psi_+}(F_k,\star)$. 
Moreover, the matrix of  this isomorphism in the bases $(a_1^{\psi_\pm},\dots,a_k^{\psi_\pm},b_1^{\psi_\pm},\dots,b_k^{\psi_\pm})$ of $H_1^{\psi_\pm}(F_k,\star)$  is 
$${\scriptsize
 \psi_+ \begin{pmatrix} 
\frac{\partial f_*(\alpha_1)}{\partial \alpha_1} & \cdots&\frac{\partial f_*(\alpha_k)}{\partial \alpha_{1}}  & \frac{\partial f_*(\beta_1)}{\partial \alpha_1} & \cdots&\frac{\partial f_*(\beta_k)}{\partial \alpha_{1}} \\
\vdots & &\vdots & \vdots & & \vdots\\
\frac{\partial f_*(\alpha_1)}{\partial \alpha_k} & \cdots&\frac{\partial f_*(\alpha_k)}{\partial \alpha_{k}}  & \frac{\partial f_*(\beta_1)}{\partial \alpha_k} & \cdots&\frac{\partial f_*(\beta_k)}{\partial \alpha_{k}}  \\
\frac{\partial f_*(\alpha_1)}{\partial \beta_1} & \cdots&\frac{\partial f_*(\alpha_k)}{\partial \beta_{1}}  & \frac{\partial f_*(\beta_1)}{\partial \beta_1} & \cdots&\frac{\partial f_*(\beta_k)}{\partial \beta_{1}} \\
\vdots & &\vdots & \vdots & & \vdots\\
\frac{\partial f_*(\alpha_1)}{\partial \beta_k} & \cdots&\frac{\partial f_*(\alpha_k)}{\partial \beta_{k}}  & \frac{\partial f_*(\beta_1)}{\partial \beta_k} & \cdots&\frac{\partial f_*(\beta_k)}{\partial \beta_{k}}  
\end{pmatrix} }
$$
where $f_*:\pi_1(F_k,\star) \to \pi_1(F_k,\star)$ is induced by $f$.
\end{lemma}

\begin{proof}
The first statement follows  from  \eqref{eq:more_general}.
The second statement is well known.
\end{proof}

\subsection{Computation of $\Alex$ with Heegaard splittings}

Assume now that $G$ is a finitely generated free abelian group and take $\F:=Q(G)$.
There are counterparts of Lemmas~\ref{lem:upper_handlebody}, \ref{lem:lower_handlebody} and \ref{lem:mapping_cylinder} for the Alexander functor $\Alex$.
These counterparts follow from the same lemmas using Theorem \ref{th:R_vs_A}, or they can be proved independently using presentations of $\Z[G]$-modules.

For $G=\{1\}$, we deduce that the functor $\Alex$ is essentially the same thing as the  TQFT constructed in \cite{FN_U(1)}.
(Compare  the formulas given in \cite[\S 3]{Kerler_HQFT} with the above lemmas.) 
However, there are a few technical differences: in particular, we have considered  surfaces with circle boundary, 
whereas \cite{FN_U(1)} works with closed surfaces.

\section{Duality}  \label{sec:duality}

We prove two duality properties for the Reidemeister functor.
In this section, $\F$ is a field where a multiplicative subgroup $G$ is fixed,
and we assume that $\F$ is equipped with an involutive automorphism $f \mapsto \overline f$ such that $\overline{g}=g^{-1}$ for all $g\in G$.

\subsection{Twisted intersection form}

The first duality satisfied by  $\Reid$ involves  the ``twisted'' intersection forms of oriented surfaces with boundary.
We start by recalling this notion.

Let $k\geq 0$ be an integer and set $\pi:=\pi_1(F_k,\star)$.
The \emph{homotopy intersection form} of~$F_k$ is the pairing
$
\lambda:\Z[\pi] \times \Z[\pi] \to \Z[\pi]
$
defined by Turaev in \cite{Turaev_intersection}. 
We also refer to  Papakyriakopoulos' work \cite{Papakyriakopoulos} where this form is implicit, 
and to Perron's work \cite{Perron} where the same form $\lambda$ is re-discovered (and is denoted there by $\omega$).

The twisted homology group $H_1(F_k,\star;\Z[\pi])$ is identified (as a left $\Z[\pi]$-module) 
to the augmentation ideal $I(\pi)$ of $\Z[\pi]$ in the following way: for any oriented loop $\gamma\subset F_k$ based at $\star$,
let $\widetilde \gamma$ be the unique lift of $\gamma$ to the universal cover of $F_k$ starting  at the preferred lift $\widetilde \star$, 
and identify $[1 \otimes \widetilde \gamma]\in H_1(F_k,\star;\Z[\pi])$ to $[\gamma]-1 \in I(\pi)$.
Thus, by restricting $\lambda$ to $I(\pi) \times I(\pi)$, we obtain a pairing
$$
\langle-,-\rangle: H_1(F_k,\star;\Z[\pi]) \times  H_1(F_k,\star;\Z[\pi]) \longrightarrow  \Z[\pi].
$$
The derivation properties of $\lambda$ given in \cite{Turaev_intersection,Perron} imply that $\langle-,-\rangle$ is \emph{sesquilinear} in the sense that
$$
\langle ax +y,z\rangle = a \langle x,z\rangle + \langle y , z \rangle, \quad \langle z, ax +y \rangle =  \langle z,x \rangle S(a) + \langle  z,y \rangle
$$
for all $a \in \Z[\pi]$ and $x,y,z \in H_1(F_k,\star;\Z[\pi])$. Here $S:\Z[\pi]\to \Z[\pi]$ is the \emph{antipode},
i.e. the $\Z$-linear map defined by $S(a)=a^{-1}$ for all $a\in \pi$.

Let now $\psi:H_1(F_k) \to G$ be a group homomorphism: this induces  a structure of right $\Z[\pi]$-module on $\F$.
By identifying $H_1^\psi(F_k,\star)$  to $\F \otimes_{\Z[\pi]} H_1(F_k,\star;\Z[\pi])$, we obtain a pairing
\begin{equation} \label{eq:twisted_pairing}
\langle-,-\rangle: H_1^\psi(F_k,\star) \times  H_1^\psi(F_k,\star) \longrightarrow  \F
\end{equation}
defined by $\langle f_1 \otimes h_1, f_2 \otimes h_2 \rangle := f_1 \overline{f_2}\, \psi(\langle h_1,h_2 \rangle)$ 
for all $f_1,f_2 \in \F$ and $h_1,h_2 \in H_1(F_k,\star;\Z[\pi])$. This pairing is \emph{sesquilinear} in the sense that
$$
\langle f x + y,z\rangle = f \langle x,z\rangle + \langle y , z \rangle, \quad \langle z, f x +y \rangle =  \overline{f} \langle z,x \rangle  + \langle  z,y \rangle
$$
for all $f\in \F$ and $x,y,z \in H_1^\psi(F_k,\star)$.
The pairing \eqref{eq:twisted_pairing} can also be defined using Poincar\'e duality (with twisted coefficients)
and the fact that $H_1^\psi(F_k,J) \simeq H_1^\psi(F_k,\star) \simeq  H_1^\psi(F_k,J')$,
where  $J,J'$ are two closed intervals such that $J \cup J'=\partial F_k $ and $J\cap J'=\partial J = \partial J'$.
In particular, the pairing \eqref{eq:twisted_pairing} is \emph{non-singular} in the sense that $\langle x,-\rangle: H_1^\psi(F_k,\star) \to \Hom(H_1^\psi(F_k,\star),\F)$
is an isomorphism for any $x\in H_1^\psi(F_k,\star)$.

For any integer  $r\geq 1$, the pairing  \eqref{eq:twisted_pairing} also  induces a non-singular sesquilinear pairing 
$\langle-,-\rangle: \Lambda^r H_1^\psi(F_k,\star) \times \Lambda^r H_1^\psi(F_k,\star) \to  \F$ defined by
$$
\langle x_1\wedge \cdots \wedge x_r ,y_1\wedge \cdots \wedge y_r   \rangle 
= \det \begin{pmatrix} \langle x_1,y_1 \rangle  & \cdots & \langle x_1,y_r \rangle  \\ \vdots & \ddots & \vdots \\ \langle x_r,y_1 \rangle  & \cdots & \langle x_r,y_r \rangle  \end{pmatrix}
$$
for all $x_1,\dots,x_r,y_1,\dots,y_r \in H_1^\psi(F_k,\star)$. For $r=0$, we set $\langle x,y\rangle := x \overline{y}$ for all $x,y \in \F$.

\begin{remark}
The sesquilinear pairing \eqref{eq:twisted_pairing}   is not skew-hermitian. Instead, we have
\begin{equation}\label{eq:not-skew}
\forall x,y \in H_1^\psi(F_k,\star), \quad \langle x, y \rangle = - \overline{\langle y, x \rangle} + \partial_*(x)\, \overline{\partial_*(y)} 
\end{equation}
where $\partial_*: H_1^\psi(F_k,\star) \to \F$ is the connecting homomorphism in the long exact sequence of the pair $(F_k,\star)$.
This identity follows from a similar property for the homotopy intersection form $\lambda$: see \cite{Turaev_intersection,Perron}.
\end{remark}

\subsection{First duality}

Let $g_-,g_+\geq 0$ be integers.
The \emph{dual} of an $M\in \Cob(g_-,g_+)$ is the cobordism $\overline M \in \Cob(g_+,g_-)$
obtained from $M$ by reversing its orientation and by composing its boundary-parametrization $m:F(g_-,g_+) \to \partial M $ with the orientation-reversing homeomorphism
$$
{\small
\underbrace{-F_{g_+} \cup_{S^1 \times \{-1\}} \big( S^1 \times [-1,1] \big) \cup_{S^1 \times \{1\}} F_{g_-}}_{F(g_+,g_-)} 
\stackrel{\cong}{\longrightarrow}  \underbrace{-F_{g_-} \cup_{S^1 \times \{-1\}} \big( S^1 \times [-1,1] \big) \cup_{S^1 \times \{1\}} F_{g_+}}_{F(g_-,g_+)}},
$$
which is given by ``time-reversal'' $(x,t)\mapsto (x,-t)$ on the annulus $S^1\times [-1,1]$ and by the identity on $F_{g_+}$ and $F_{g_-}$.

\begin{theorem} \label{th:duality1}
For any $(M,\varphi)\in \Cob_G\big((g_-,\varphi_-),(g_+,\varphi_+)\big)$ and for any  $j\geq 0$, we have
\begin{equation} \label{eq:adjunction}
\forall x \in \Lambda^j H_-, \ \forall y \in \Lambda^{j+\delta\! g}H_+, \quad
{ \big\langle \Reid(M,\varphi)(x),y  \big\rangle} =   \left\langle x , \Reid\left(\overline M,\varphi\right )(y)  \right\rangle
\end{equation}
where $\delta\! g:= g_+-g_-$ and $H_\pm := H_1^{\varphi_\pm}(F_{g_\pm},\star)$.
\end{theorem}

\noindent
Of course, the identity \eqref{eq:adjunction} only holds true up to multiplication by a constant in $\pm G$ (independent of $x$ and $y$).
The pairing $\langle -,-\rangle$ denotes the twisted intersection form of $H_+$ (respectively, $H_-$) on the left-hand side (respectively, the right-hand side)  of \eqref{eq:adjunction}.

\begin{proof}[Proof of Theorem \ref{th:duality1}]
Assume that $(M,\varphi)= (M',\varphi') \circ (M'',\varphi'')$ where  $(M',\varphi')$ and $(M'',\varphi'')$ are two morphisms in $\Cob_G$ satisfying \eqref{eq:adjunction}.
Then the dual of $M$ is $\overline M'' \circ \overline M'$, and it easily follows that $(M,\varphi)$ also satisfies \eqref{eq:adjunction}.
Consequently, and following the discussion of \S \ref{sec:Heegaard},
it is enough to prove \eqref{eq:adjunction} in the following three cases: (i)~$M$ is a mapping cylinder; (ii)~$M$ is a ``stabilized'' lower handlebody; (iii)~$M$ is a ``stabilized'' upper handlebody.\\

{\it Case (i).} Assume that $g_-=g_+=:k$ and that $M= \mcyl(f)$ is the mapping cylinder of an $f \in \mcg(F_k)$.
Then $\overline M= \mcyl(f^{-1})$.
Since $\varphi_+f=\varphi_-:H_1(F_k) \to G$ and since $f_*:\pi_1(F_k,\star)\to \pi_1(F_k,\star)$ preserves the homotopy intersection form,
the isomorphism $f:H_- \to H_+$ induced by $f:F_k\to F_k$ preserves the pairings \eqref{eq:twisted_pairing}.
Using the first statement of Lemma \ref{lem:mapping_cylinder}, we obtain \eqref{eq:adjunction}  as follows:
\begin{eqnarray*}
\forall x\in \Lambda^j H_-, \ \forall y \in \Lambda^j H_+, \quad
{ \left\langle \Reid(M,\varphi)(x),y  \right\rangle}&=&  \left\langle \Lambda^j f(x),y \right\rangle \\
&=&  \left\langle x, \Lambda^j f^{-1}(y) \right\rangle \ = \  \left\langle x , \Reid\left(\overline M,\varphi\right )(y)\right\rangle.\\[-0.3cm]
\end{eqnarray*}

{\it Interlude.}
In order to deal with cases (ii) and (iii), we need an explicit formula for the twisted intersection form $\langle-,-\rangle: H_1^\psi(F_k,\star) \times H_1^\psi(F_k,\star)  \to \F$
defined by  a group homomorphism $\psi: H_1(F_k) \to G$.
For this, we fix a system of meridians and parallels $(\alpha_1,\dots,\alpha_k, \beta_1,\dots,\beta_k)$ on $F_k$ as explained in \S \ref{subsec:Heegaard},
and we denote by $(a_1^{\psi},\dots, a_k^{\psi},b_1^{\psi},\dots,b_k^{\psi})$ the corresponding  basis of $H_1^{\psi}(F_k,\star)$: see \eqref{eq:a_b}.
For every $x,y \in H_1(F_k)$,  set $P^\psi(x,y):= (1-\psi(x))\overline{(1-\psi(y))} \in \F$.
Then, for an appropriate choice of meridians and parallels, the matrix of $\langle -,-\rangle$ in the basis  $(a_1^{\psi},\dots, a_k^{\psi},b_1^{\psi},\dots,b_k^{\psi})$ is
$$
J^{\psi} = \left(\begin{array}{c|c} J_{aa}^{\psi} & J_{ab}^{\psi} \\ \hline  J_{ba}^{\psi} & J_{bb}^{\psi} \end{array}\right)
$$
where $J_{aa}^{\psi}, J_{ab}^{\psi}, J_{ba}^{\psi}, J_{bb}^{\psi}$ are the following lower triangular matrices  \cite[Lemma 2.4]{Perron}:
\begin{equation}\label{eq:aa}
{\scriptsize J_{aa}^{\psi} = \begin{pmatrix}
{1 - \psi(\alpha_1)} &       0  &    0   &         \cdots  & 0  \\
P^\psi(\alpha_2,\alpha_1) & 1 - \psi(\alpha_2) &   0    &       \cdots   &   0 \\
P^\psi(\alpha_3,\alpha_1)  & P^\psi(\alpha_3,\alpha_2) & \ddots &     \ddots      &   \vdots  \\
\vdots  & \vdots  & \ddots & \ddots    &  0   \\
P^\psi(\alpha_k,\alpha_1)  & P^\psi(\alpha_k,\alpha_2)  & \ldots & P^\psi(\alpha_k,\alpha_{k-1})  & 1 - \psi(\alpha_k)
\end{pmatrix}},
\end{equation}
\begin{equation} \label{eq:ab}
{\scriptsize J_{ab}^{\psi} =  \begin{pmatrix}
\psi(\alpha_1) \overline{\psi(\beta_1)} &       0  &    0   &         \cdots  & 0  \\
P^\psi(\alpha_2,\beta_1) & \psi(\alpha_2) \overline{\psi(\beta_2)}   &   0   &       \cdots   &   0 \\
P^\psi(\alpha_3,\beta_1) & P^\psi(\alpha_3,\beta_2) & \ddots &     \ddots      &   \vdots  \\
\vdots  & \vdots  & \ddots & \ddots    &  0   \\
P^\psi(\alpha_k,\beta_1)& P^\psi(\alpha_k,\beta_2) & \ldots & P^\psi(\alpha_k,\beta_{k-1})& \psi(\alpha_k) \overline{\psi(\beta_k)} 
\end{pmatrix}}, 
\end{equation}
\begin{equation} \label{eq:ba}
 {\scriptsize J_{ba}^{\psi} =  \begin{pmatrix}
1- \overline{\psi(\alpha_1)} -\psi(\beta_1) &       0  &    0   &         \cdots  & 0  \\
P^\psi(\beta_2,\alpha_1) & 1- \overline{\psi(\alpha_2)} -\psi(\beta_2)  &   0    &       \cdots   &   0 \\
P^\psi(\beta_3,\alpha_1) & P^\psi(\beta_3,\alpha_2) & \ddots &     \ddots      &   \vdots  \\
\vdots  & \vdots  & \ddots & \ddots    &  0   \\
P^\psi(\beta_k,\alpha_1) & P^\psi(\beta_k,\alpha_2) & \ldots & P^\psi(\beta_k,\alpha_{k-1}) & 1- \overline{\psi(\alpha_k)} -\psi(\beta_k) 
\end{pmatrix},}
\end{equation}
\begin{equation} \label{eq:bb}
{\scriptsize J_{bb}^{\psi} =  \begin{pmatrix}
1- \overline{\psi(\beta_1)} &       0  &    0   &         \cdots  & 0  \\
P^\psi(\beta_2,\beta_1) & 1- \overline{\psi(\beta_2)} &   0    &       \cdots   &   0 \\
P^\psi(\beta_3,\beta_1)  & P^\psi(\beta_3,\beta_2)  & \ddots &     \ddots      &   \vdots  \\
\vdots  & \vdots  & \ddots & \ddots    &  0   \\
P^\psi(\beta_k,\beta_1)  & P^\psi(\beta_k,\beta_2)  & \ldots & P^\psi(\beta_k,\beta_{k-1})  & 1- \overline{\psi(\beta_k)}
\end{pmatrix}.}
\end{equation}

Besides, the following notation will be useful in the sequel. Let $\varepsilon\in \{+,-\}$ be a sign.
We  denote by $(v_1^\varepsilon,\dots,v_{g_\varepsilon}^\varepsilon, v_{g_\varepsilon +1}^\varepsilon, \dots, v_{2g_\varepsilon}^\varepsilon):= 
(a_1^{\varphi_{\varepsilon}},\dots,a_{g_\varepsilon}^{\varphi_{\varepsilon}},b_1^{\varphi_{\varepsilon}},\dots,b_{g_\varepsilon}^{\varphi_{\varepsilon}})$ 
the basis of $H_\varepsilon=H_1^{\varphi_\varepsilon}(F_{g_\varepsilon},\star)$.
For any $s$-element subset $P \subset \{1,\dots, 2g_\varepsilon \}$, let $v_P^\varepsilon \in \Lambda^s H_\varepsilon$ be the wedge of the vectors $v_p^\varepsilon$'s for all $p\in P$
and, when this makes sense, 
we shall also  denote by $(v_{P}^{\varepsilon})^{-\varepsilon} \in \Lambda^s H_{-\varepsilon}$  the multivector obtained from $v_P^\varepsilon$ 
by the transformations $a_i^{\varphi_{\varepsilon}} \mapsto a_{i-\varepsilon \delta\! g }^{\varphi_{-\varepsilon}}$ and $b_i^{\varphi_{\varepsilon}} \mapsto b_{i-\varepsilon \delta\! g}^{\varphi_{-\varepsilon}}$.\\

{\it Case (ii).} Assume  that  $M=C_0^{r} \otimes \Id_{g_-}$ where $r=\delta\! g$. 
Note that $\varphi_+(\alpha_i)=1$ for all $i\in\{1,\dots,r\}$, so that \eqref{eq:aa} and \eqref{eq:ab} applied to  $\psi:=\varphi_+$ give
\begin{equation} \label{eq:<a,...>}
\forall i  \in \{1,\dots,r\}, \forall j \in \{1,\dots, r+g_-\}, \quad \langle a_i^{\varphi_+}, a_j^{\varphi_+} \rangle =0 , \  \langle a_i^{\varphi_+}, b_j^{\varphi_+} \rangle = \delta_{ij}\, \overline{\varphi_+(\beta_j)} 
\end{equation}
and, combining this with \eqref{eq:not-skew}, we also obtain
\begin{equation} \label{eq:<...,a>}
\forall i  \in \{1,\dots,r\}, \forall j \in \{1,\dots, r+g_-\}, \quad \langle a_j^{\varphi_+}, a_i^{\varphi_+} \rangle =0 , \  \langle b_j^{\varphi_+}, a_i^{\varphi_+} \rangle = - \delta_{ij}\, {\varphi_+(\beta_j)}.
\end{equation}

Let $P \subset \{1,\dots, 2g_-\}$ with $ \vert P \vert =j$ and let $Q \subset \{1,\dots, 2g_+\}$ with $\vert Q \vert = r+j$.
It follows from Lemma \ref{lem:lower_handlebody} that
$$
{ \big\langle \Reid(M,\varphi)(v_P^-), v_Q^+  \big\rangle}  = { \big\langle a_1^{\varphi_+} \wedge \cdots \wedge a_r^{\varphi_+} \wedge (v_{P}^-)^+, v_Q^+ \big\rangle}.
$$
According to \eqref{eq:<a,...>}, this determinant is zero if the subset $B:= \{g_++1,\dots,g_++r\}$ is not contained in $Q$. If $B \subset Q$, then we get
\begin{eqnarray*}
{ \big\langle \Reid(M,\varphi)(v_P^-), v_Q^+  \big\rangle}  &=&  \varepsilon_{B} { \left\langle a_1^{\varphi_+} \wedge \cdots \wedge a_r^{\varphi_+} \wedge (v_{P}^-)^+, v_B^+ \wedge v_{B^c}^+ \right\rangle} \\
&=&   \varepsilon_{B} \big\langle a_1^{\varphi_+} \wedge \cdots \wedge a_r^{\varphi_+} , v_B^+ \big\rangle\, \big\langle (v_P^-)^+, v_{B^c}^+ \big\rangle \\
& = &   \varepsilon_{B}\,  \overline{\varphi_+(\beta_1 \cdots \beta_r)}\,  \big\langle (v_P^-)^+, v_{B^c}^+ \big\rangle
\end{eqnarray*}
where $B^c := Q \setminus B$ and $\varepsilon_{B}$ is the signature of the permutation $B B^c$ (where the elements of $B$ in increasing order are followed by the elements of $B^c$ in increasing order).
We also deduce from \eqref{eq:<...,a>} that $\langle (v_P^-)^+, v_{B^c}^+ \rangle =0$ if $B^c$ has a non-empty intersection with $A:=\{1,\dots,r\}$, 
and it follows that ${ \big\langle \Reid(M,\varphi)(v_P^-), v_Q^+  \big\rangle} =0$ if $ A \cap Q \neq \varnothing$. 

Besides, it follows from Lemma  \ref{lem:upper_handlebody} that $\Reid\left(\overline M,\varphi\right )(v_Q^+)$ is trivial if $A \cap Q \neq \varnothing$ or $B$ is not contained in $Q$. 
If $A \cap Q = \varnothing$ and $B \subset Q$, we get
$$
 \big\langle v_P^- , \Reid\left(\overline M,\varphi\right )(v_Q^+)  \big\rangle  =   \varepsilon_{B}\,   \left\langle v_P^- , \Reid\left(\overline M,\varphi\right )(v_B^+ \wedge v_{B^c}^+)  \right\rangle \\
  =   \varepsilon_{B}\,   \big\langle v_P^- ,   (v_{B^c}^+)^- \big\rangle. 
$$
We deduce that ${ \big\langle \Reid(M,\varphi)(v_P^-), v_Q^+  \big\rangle}  = \overline{\varphi_+(\beta_1 \cdots \beta_r) }\big\langle v_P^- , \Reid\left(\overline M,\varphi\right )(v_Q^+)  \big\rangle$ for any $P,Q$.
Since \eqref{eq:adjunction} is only required to hold true up to multiplication by a constant in $\pm G$, the theorem is proved in case (ii).\\

{\it Case (iii).}  Assume now that  $M=C_r^{0} \otimes \Id_{g_+}$ where $r= -\delta\! g$. 
Note that $\varphi_-(\beta_i)=1$ for all $i\in \{1,\dots,r\}$, so that \eqref{eq:bb} and \eqref{eq:ab} applied to $\psi:=\varphi_-$ give 
\begin{equation} \label{eq:<...,b>}
\forall i  \in \{1,\dots,r+g_+\}, \forall j \in \{1,\dots, r\}, \quad \langle b_i^{\varphi_-}, b_j^{\varphi_-} \rangle =0 , \  \langle a_i^{\varphi_-}, b_j^{\varphi_-} \rangle = \delta_{ij}\, {\varphi_-(\alpha_i)} 
\end{equation}
and, combining this with \eqref{eq:not-skew}, we also obtain
\begin{equation} \label{eq:<b,...>}
\forall i  \in \{1,\dots,r+g_+\}, \ j \in \{1,\dots, r\}, \quad \langle b_j^{\varphi_-}, b_i^{\varphi_-} \rangle =0 , \  \langle b_j^{\varphi_-}, a_i^{\varphi_-} \rangle = -\delta_{ij}\, \overline{\varphi_-(\alpha_i)} .
\end{equation}

Let $P \subset \{1,\dots, 2g_-\}$ with $ \vert P \vert =j$ and let $Q \subset \{1,\dots, 2g_+\}$ with $\vert Q \vert = j-r$.
By Lemma \ref{lem:upper_handlebody}, $\Reid(M,\varphi)(v_P^-)$ is trivial if $P$ does not contain $A:=\{1,\dots,r\}$ or $P$ has a non-empty intersection with $B:= \{g_-+1,\dots,g_-+r\}$. 
If $A \subset P$ and $P \cap B = \varnothing$, we obtain
$$
 \big\langle \Reid(M,\varphi)(v_P^-), v_Q^+  \big\rangle = \varepsilon_{A}\, \big\langle \Reid(M,\varphi)(v_A^- \wedge v_{A^c}^-), v_Q^+  \big\rangle = \varepsilon_{A}\, \big\langle  (v_{A^c}^-)^+, v_Q^+  \big\rangle
$$
where $A^c:= P \setminus A$ and $\varepsilon_A$ is the signature of the permutation $AA^c$.

Besides,  Lemma~\ref{lem:lower_handlebody} gives
$$
 \big\langle v_P^- , \Reid\left(\overline M,\varphi\right )(v_Q^+)  \big\rangle  = \big\langle v_P^-, b_1^{\varphi_-}\wedge \cdots \wedge b_r^{\varphi_-} \wedge  (v_Q^+)^-\big\rangle
$$
which, according to \eqref{eq:<...,b>}, is trivial if $P$ does not contain $A$. If $A\subset P$, we get
\begin{eqnarray*}
 \big\langle v_P^- , \Reid\left(\overline M,\varphi\right )(v_Q^+)  \big\rangle  &=   & \varepsilon_{A}\, \big\langle v_A^- \wedge v_{A^c}^-, b_1^{\varphi_-}\wedge \cdots \wedge b_r^{\varphi_-} \wedge  (v_Q^+)^-\big\rangle \\
  &=   & \varepsilon_{A}\, \langle v_A^- , b_1^{\varphi_-} \wedge \cdots \wedge b_r ^{\varphi_-}\rangle\,  \big\langle v_{A^c}^-, (v_Q^+)^-  \big\rangle \\
  & = &   \varepsilon_{A}\,  \varphi_-(\alpha_1 \cdots \alpha_r)\,  \big\langle v_{A^c}^-,(v_Q^+)^-  \big\rangle.
\end{eqnarray*}
It follows from \eqref{eq:<b,...>} that  $\big\langle v_{A^c}^-,(v_Q^+)^-  \big\rangle=0$ if $A^c$ has a non-empty intersection with $B$, 
so that $\big\langle v_P^- , \Reid\left(\overline M,\varphi\right )(v_Q^+)  \big\rangle=0$ if $P \cap B \neq \varnothing$.
We deduce that $\big\langle \Reid(M,\varphi)(v_P^-), v_Q^+  \big\rangle  = \overline{\varphi_-(\alpha_1 \cdots \alpha_r)}\,   \big\langle v_P^- , \Reid\left(\overline M,\varphi\right )(v_Q^+)  \big\rangle$ for any $P,Q$.
This proves the theorem in case (iii).
\end{proof}

\begin{example} \label{ex:sym_hcob}
We consider the situation of \S \ref{subsec:R_C}: let $\psi: H_1(F_k) \to G$ be a group homomorphism and let $M\in \C^\psi(F_k)$ with $k\geq 1$.
According to Proposition \ref{prop:homology_cobordisms_Reidemeister}, 
$\Reid(M,\psi)$ is determined by the relative Reidemeister torsion $\tau^\psi(M,\partial_+M)$ and the Magnus representation $r^\psi(M):H^\psi \to H^\psi$, where $H^\psi := H_1^\psi(F_k,\star)$.
Specializing Theorem~\ref{th:duality1} to $j:=0$, we obtain the well-known duality theorem
\begin{equation}\label{eq:duality_hcob}
\tau^\psi(M,\partial_+M) =   \overline{\tau^\psi(M,\partial_-M)} \ \in \F/\pm G,
\end{equation}
see \cite[Appendix 3]{Turaev_RTKT}. Next, specializing Theorem~\ref{th:duality1} successively to $j:=1$ and $j:=2$, we  obtain the invariance property
$$
\forall x,z\in H^\psi, \quad \big\langle r^\psi(M)(x), r^\psi(M)(z)\big\rangle = \left\langle x,z \right\rangle,
$$
which is already observed in  \cite[Theorem 2.4]{Sakasai_symplecticity}.
\end{example}

\begin{example} \label{ex:sym_knot}
We consider the situation of \S \ref{subsec:knot_exteriors}: let $G$ be the infinite cyclic group generated by $t$ and $\F:= Q(G)$,
let $M_K$ be the exterior of an oriented knot $K$ in an oriented homology $3$-sphere and let $\varphi_K: H_1(M_K) \to G$ be the canonical isomorphism.
There is  a system of meridian and parallel  $(\alpha,\beta)$ on $F_1$ 
and  a boundary-parametrization $m:F(1,0) \to \partial M_K$  such that 
\begin{itemize}
\item[(i)] $m_-(\alpha)$ is the oriented meridian of $K$ and   $m_-(\beta)$ is  \emph{the} parallel of $K$ that is null-homologous in $M_K$,
\item[(ii)] the matrix of $\langle -,-\rangle: H_- \times H_- \to \F$ in the corresponding basis  $(a,b):= (a_1^{\varphi_K m_-},b_1^{\varphi_K m_-}) $ of  $H_-:=H_1^{\varphi_K m_-}(F_1,\star)$  is
$ \begin{pmatrix} 1-t & t \\ -t^{-1} & 0 \end{pmatrix}.$
\end{itemize}
According to Proposition \ref{prop:Alexander_knot}, the map $\Reid(M_K,\varphi_K)$ is determined by the Alexander polynomial $\Delta(K)$.
By applying Theorem \ref{th:duality1} successively to  $x:= a$ and $x:=b$, we get
\begin{equation} \label{eq:M_K_bar}
\Reid(\overline{M_K}, \varphi_K)(1) = \overline{\Delta(K)}\,  b \ \in H_-.
\end{equation}
\end{example}

\subsection{Second duality}

The second duality satisfied by $\Reid$ does not involve the conjugation $f\mapsto \overline{f}$ of the field $\F$, 
and it is an immediate consequence of the definitions.

\begin{proposition} \label{prop:duality2}
For any $(M,\varphi)\in \Cob_G((g_-,\varphi_-),(g_+,\varphi_+))$ and  $j\geq 0$, we have
$$
\forall x \in \Lambda^j H_-, \ \forall y \in \Lambda^{g-j}H_+, \quad
\vol\big(\Reid(M,\varphi)(x) \wedge y \big) =   (-1)^{jg} \cdot \vol\left( x \wedge \Reid(\overline M,\varphi)(y) \right) 
$$
where $g:= g_++g_-$, $H_\pm := H_1^{\varphi_\pm}(F_{g_\pm},\star)$ and $\vol : \Lambda^{2g_\pm} H_\pm \to \F$ is an arbitrary integral volume form.
\end{proposition}

Despite its simplicity, this proposition turns out to be interesting when it is combined with Theorem \ref{th:duality1}.

\begin{example}
We use the same notation as in Example \ref{ex:sym_hcob}. Let $(z_1,\dots,z_{2k})$ be a basis of $H^\psi$ 
arising from of a basis of the free $\Z[H_1(F_k)]$-module $H_1(F_k,\star;\Z[H_1(F_k)])$ and assume that $\vol$ is given by $\vol(z_1\wedge \cdots \wedge z_{2k})=1$. 
By applying Proposition \ref{prop:duality2} to $x:=z_1\wedge \cdots \wedge z_{2k}$, we get
$
\tau^\psi(M,\partial_+ M)\cdot  \det r^\psi(M) = \tau^\psi(M,\partial_-M).
$
Combined with \eqref{eq:duality_hcob}, this relation gives the symmetry
$$
\tau^\psi(M,\partial_+ M) \cdot  \det r^\psi(M) = \overline{\tau^\psi(M,\partial_+M)} \ \in \F/\pm G
$$
which is also observed in \cite[Theorem 5.3]{Sakasai_Magnus2}.
\end{example}

\begin{example}
We use the same notation as in Example \ref{ex:sym_knot}. Let   $\vol$ be the volume form on $H_-$ defined by $\vol(a\wedge b)= 1$.
By applying Proposition \ref{prop:duality2}  successively  to $x:=a$ and $x:=b$, we obtain $\Reid(\overline{M_K},\varphi_K)(1)= \Delta(K)\, b$.
Combined with \eqref{eq:M_K_bar}, we recover the classical symmetry of the Alexander polynomial:
$$
\Delta(K) = \overline{\Delta(K)} \in \Z[G]/\pm G.
$$
\end{example}

 \appendix

\section{A short review of combinatorial torsions} \label{sec:review_torsion}

We recall the definition and basic properties of the torsions of chain complexes. 
The reader is referred to \cite{Milnor_Whitehead} and \cite{Turaev_book} for further details and references.
In this appendix, $\F$ is a field.

\subsection{Definition of the torsion} \label{subsec:torsion_def}

Given an $\F$-vector space $V$ of finite dimension $n\geq 0$, 
an $n$-tuple $b=(b_1,\dots,b_n)$  of  vectors  in $V$ and a basis $c=(c_1,\dots,c_n)$ of $V$, 
we denote by 
$
[b/c]\in \F
$ 
the determinant of the matrix expressing $b$ in the basis $c$. 
Two bases $b$ and $c$ are said to be \emph{equivalent} if $[b/c]=1$.

Given a short exact sequence of $\F$-vector spaces $0 \to V' \to V \to V'' \to 0$
and some bases $c'$ and $c''$ of $V'$ and $V''$ respectively, we  denote by $c' c''$
the equivalence class of bases of $V$ obtained by juxtaposing (in this order) the image of $c'$  in $V$ and a lift of $c''$ to $V$.

By a \emph{finite $\F$-chain complex of length} $m\geq 1$, we mean  a  chain complex $C$ in the category of finite-dimensional $\F$-vector spaces
and we assume that $C$ is concentrated in degrees $0,\dots,m$:
$$
C=\Big( \xymatrix{ C_m\ar[r]^-{\partial_{m}} & C_{m-1} \ar[r]  & \cdots \ar[r]^{\partial_1} & C_0} \Big).
$$
A \emph{basis} of $C$ is a family  $c=(c_m,\dots,c_0)$ where $c_i$ is a basis of  $C_i$ for all $i\in\{0,\dots,m\}$.
A  \emph{homological basis} of $C$ is a family $h=(h_m,\dots,h_0)$ where $h_i$ is a basis of the $i$-th homology group $H_i(C)$ for all $i\in\{0,\dots,m\}$.
If we have choosen a basis $b_j$ of the space of $j$-dimensional boundaries $B_j(C):= \operatorname{Im} \partial_{j+1}$ for all $j\in \{0,\dots,m-1\}$,
then a homological basis $h$ of $C$ induces an equivalence class of bases of $C_i$ for any $i$:
specifically, we consider  the basis   $(b_ih_i)b_{i-1}$ of $C_i$ obtained by juxtaposition in   the following short exact sequences where we denote  $Z_i(C):=\Ker \partial_i$:
\begin{eqnarray}
\label{eq:BZH} &&0  \longrightarrow B_i(C) \longrightarrow Z_i(C) \longrightarrow H_i(C) \longrightarrow 0\\
\label {eq:ZCB} &\hbox{and}& 0  \longrightarrow Z_i(C) \longrightarrow C_i \stackrel{\partial_i} {\longrightarrow} B_{i-1}(C) \longrightarrow 0.
\end{eqnarray}

\begin{definition} 
\label{def:algebraic_torsion}
The \emph{torsion} of  a finite $\F$-chain complex $C$ of length $m$, equipped with a basis $c$ and a homological basis $h$, is the scalar
\begin{displaymath} 
\tau(C;c,h):=  \prod_{i=0}^m \big[(b_ih_i)b_{i-1}/c_i\big]^{(-1)^{i+1}} \in \F\setminus \{0\}.
\end{displaymath}
It is easily checked that this definition does not depend on the choice of $b_0,\dots,b_m$ 
and, when $C$ is acyclic, we set $\tau(C;c):= \tau(C;c,\varnothing)$.
\end{definition}

The following lemma, which is well known, is a way of viewing the torsion  as a function in homology. 

\begin{lemma}\label{lem:torsion_as_function}
Let $C$ be a finite $\F$-chain complex of length $m\geq 1$,
let $k\in \{0,\dots,m\}$ and set $\beta:= \dim H_k(C)$.
Assume given a basis $c=(c_m,\dots,c_0)$ of $C$ and a basis $h_i$ of $H_i(C)$ for every $i\neq k$.
Then there is a unique linear map $\ell:\Lambda^{\beta} H_k(C) \to \F$ such that
$$
\ell(v_1\wedge \cdots \wedge v_{\beta}) = 
\left\{ \begin{array}{ll}
\tau\big(C;c,(h_m,\dots,h_{k+1},v,h_{k-1},\dots,h_0) \big) & \hbox{if $k$ is odd},  \\
\tau\big(C;c,(h_m,\dots,h_{k+1},v,h_{k-1},\dots,h_0) \big)^{-1} & \hbox{if $k$ is even},  
\end{array}\right.
$$
for any basis $v=(v_1,\dots, v_{\beta})$ of $H_k(C)$.
\end{lemma}

\begin{proof}
The unicity of $\ell$ is obvious and, clearly, we can assume that $k$ is odd. 
Let $s:H_k(C) \to Z_k(C)$ and $t:B_{k-1}(C) \to C_k$ be $\F$-linear sections of \eqref{eq:BZH} and \eqref{eq:ZCB}, respectively.
For any $\beta$-tuple $v=(v_1,\dots, v_{\beta})$ of elements of $H_k(C)$, we set
$$
\ell(v) :=   \big[ b_k\, s(v)\, t(b_{k-1}) / c_k\big] \cdot \prod_{i\neq k }\big[(b_ih_i)b_{i-1}/c_i\big]^{(-1)^{i+1}} \in \F
$$
where $b_k\, s(v)\, t(b_{k-1})$ denotes the family of vectors of $C_k$ obtained by juxtaposing (in this order) $b_k$, $s(v)$ and $t(b_{k-1})$.
The resulting map $\ell: H_k(C)^{\beta}\to \F$ is multilinear and alternate, 
hence it induces a map $\ell: \Lambda^{\beta} H_k(C)\to \F$ with the desired property.
\end{proof}

\subsection{Multiplicativity of the torsion}

Consider a short exact sequence of finite $\F$-chain complexes of length $m\geq 1$:
\begin{equation}
\label{eq:another_ses}
\xymatrix{
0 \ar[r] & C'  \ar[r] & C \ar[r] & C'' \ar[r] & 0.
}
\end{equation}
Let us assume that $C', C,C''$ are based by $c', c,c''$ respectively,
and homologically based by $h', h,h''$ respectively. 
We further assume that the bases $c', c,c''$ are \emph{compatible} in the sense that
$c_i$ is equivalent to $c'_ic''_i$ for every $i\in \{0,\dots,m\}$.
The short exact sequence (\ref{eq:another_ses}) induces a long exact sequence in homology:
$$
\mathcal{H} := \big(H_m(C') \to H_m(C) \to  H_m(C'') \to  \cdots \to H_0(C') \to H_0(C) \to H_0(C'')\big).
$$
We regard  $\mathcal{H}$ as an acyclic finite $\F$-chain complex based by 
$$
(h',h,h''):= (h'_m,h_m,h''_m,\dots,h'_0,h_0,h''_0).
$$

The following formula is classical in the theory of combinatorial torsions: see \cite[Theorem 3.2]{Milnor_Whitehead} or \cite[Lemma 3.4.2]{Turaev_RTKT}.

\begin{theorem} 
\label{th:multiplicativity}
With the above notation, we have
\begin{equation}
\label{eq:multiplicativite}
\tau(C;c,h)=  \varepsilon \cdot  \tau(C';c',h') \cdot \tau(C'';c'',h'') \cdot \tau\big(\mathcal{H};  (h',h,h'')\big)
\end{equation}
where $\varepsilon$ is a sign  depending only on the dimensions of the $\F$-vector spaces $C'_i,C_i,C''_i$ and $H_i(C'),H_i(C),H_i(C'')$ for all $i\in \{0,\dots,m\}$.
\end{theorem}

\begin{example} \label{ex:direct_sum}
Assume that $C=C'\oplus C''$ and that the chain maps $C'\to C$ and $C\to C''$ in \eqref{eq:another_ses} are the natural inclusion and projection, respectively.
For all $i\in \{0,\dots,m\}$, let $c_i$ be the basis of $C_i=C'_i\oplus C''_i$ obtained by juxtaposing (in this order) some  bases $c'_i$ and $c''_i$ of $C'_i$ and $C''_i$, respectively;
similarly, let $h_i$ be the basis of $H_i(C) = H_i(C')\oplus H_i(C'')$ obtained by juxtaposing some bases $h'_i$ and $h''_i$ of $H_i(C')$ and $H_i(C'')$, respectively.
We set  $c:=(c_m,\dots,c_0)$ and $h:=(h_m,\dots,h_0)$.
Then $\tau(C;c,h)=  \varepsilon \cdot  \tau(C';c',h') \cdot \tau(C'';c'',h'')$.
\end{example}

\bibliographystyle{amsalpha}

\bibliography{AF}

\end{document}